\documentclass[11pt,reqno]{amsart}

\usepackage[utf8]{inputenc}
\usepackage[T1]{fontenc}
\usepackage{amsmath,amssymb,amsthm,mathtools}
\usepackage{enumitem}
\usepackage{graphicx}
\usepackage{booktabs}
\usepackage{tikz}
\usetikzlibrary{arrows.meta}
\usepackage[a4paper,top=2.6cm,bottom=2.6cm,left=2.8cm,right=2.8cm]{geometry}
\usepackage[expansion=false]{microtype}

\newcommand{\EE}{\mathbb{E}}
\newcommand{\PP}{\mathbb{P}}

\newcommand{\RR}{\mathbb{R}}
\newcommand{\ZZ}{\mathbb{Z}}

\newcommand{\pen}{\mathrm{pen}}

\newcommand{\Q}{\mathcal Q}
\newcommand{\hq}{h_q}
\newcommand{\hsh}{h^{\sharp}}
\newcommand{\hstar}{h^{*}}
\newcommand{\qsh}{q^{\sharp}}
\newcommand{\mstar}{m^{*}}
\newcommand{\mfloor}{m_{\star}}
\newcommand{\Chat}{\widehat C}
\newcommand{\ghat}{\widehat\gamma}
\newcommand{\Khat}{\widehat K}
\newcommand{\chat}{\widehat c}

\newcommand{\norm}[1]{\lVert #1\rVert}
\newcommand{\inner}[2]{\langle #1,#2\rangle}
\newcommand{\Eff}{D_{\mathrm{eff}}}
\newcommand{\ind}{\mathbf 1}
\DeclareMathOperator{\Var}{Var}
\DeclareMathOperator{\Cov}{Cov}
\DeclareMathOperator{\tr}{tr}
\DeclareMathOperator*{\argmin}{arg\,min}
\DeclareMathOperator{\rank}{rank}
\DeclareMathOperator{\supp}{supp}
\DeclareMathOperator{\HS}{HS}
\DeclareMathOperator{\op}{op}
\DeclareMathOperator{\KL}{KL}
\DeclareMathOperator{\sym}{sym}

\theoremstyle{plain}
\newtheorem{theorem}{Theorem}[section]
\newtheorem{proposition}[theorem]{Proposition}
\newtheorem{lemma}[theorem]{Lemma}
\newtheorem{corollary}[theorem]{Corollary}
\newtheorem*{ulemma}{Lemma}
\newtheorem*{ucorollary}{Corollary}
\theoremstyle{definition}
\newtheorem{definition}[theorem]{Definition}
\newtheorem{assumption}{Assumption}
\newtheorem{condition}[theorem]{Condition}
\theoremstyle{remark}
\newtheorem{remark}[theorem]{Remark}

\usepackage[
colorlinks=true,
linkcolor=blue,
citecolor=darkgray,
urlcolor=blue,
bookmarks=true,
bookmarksnumbered=true,
unicode=true
]{hyperref}
\hypersetup{
	pdftitle={How far can symmetry help? Phase transitions and symmetry selection in sparse functional data analysis},
	pdfauthor={Jocelyn Nembe},
	pdfsubject={Statistics - sparse functional data},
	pdfkeywords={sparse functional data, phase transition, group symmetry, equivariance, covariance surface, minimax rates, model selection}
}

\begin{document}
	
	\title[Symmetry and phase transitions in sparse FDA]{How far can symmetry help? Phase transitions and symmetry selection in sparse functional data analysis}
	
	\author{Jocelyn Nemb\'e}
	\address{Laboratory of Engineering Applied to Business Management  //  P.O. BOX 190, LIBREVILLE}
	\email{jnembe@hotmail.com}
	
	\subjclass[2020]{Primary 62G05; secondary 62R10, 62G08, 62H25}
	\keywords{sparse functional data, phase transition, group symmetry, equivariance, covariance surface, minimax rates, model selection}
	
	\begin{abstract}
		In sparse functional data analysis the covariance surface undergoes a sharp
		phase transition at a critical sampling intensity $\mstar_n\asymp n^{1/(2\beta)}$.
		We determine what a symmetry of the domain does to that transition. A cyclic
		group of order $q$ leaving the process and the design invariant displaces the
		threshold to $n^{1/(2\beta)}q^{-1/2}$; the exponent is a square root because
		symmetry acts on the number of usable pairs, which enters the variance
		quadratically, while the parametric floor of order $n^{-1}$ is left untouched by
		group averaging. The displacement saturates: once the orbit is finer than the
		bandwidth the reduction factor is $\min(q,c_K/h)$, a consequence of Poisson
		summation and of the positive definiteness of the kernel autocorrelation, and
		beyond that point the rate collapses to the one-dimensional nonparametric rate.
		Consequently no rotation symmetry, including the full circle group, lowers the
		threshold below $n^{1/(4\beta)}$; this floor is attained from above by the
		proposed estimator, and from below up to a polynomial factor by the lower bounds
		of Section~\ref{sec:minimax}. Equivalently, symmetry takes one exactly halfway
		on a logarithmic scale from the classical threshold to constant sampling. Lower
		bounds within a small polynomial factor are obtained, the uniform one by a
		positivity-preserving packing of the stationary sub-class; closing the gap
		between upper and lower bounds is posed as an open problem
		(Remark~\ref{rem:lb-gap}). Dropping the assumption that the symmetry is correct,
		the risk acquires an approximation term and the design plane splits into three
		regimes, one of which is unreachable by additional sampling; the symmetry
		spectrum governing it is explicit in Fourier coordinates, and a hold-out
		procedure selects the symmetry level with leading constant one below saturation,
		and within an absolute constant beyond; both regimes, and the laws above, are
		confirmed numerically. By contrast, reparametrisation of the domain leaves the
		threshold exactly where it was.
	\end{abstract}
	
	\maketitle
	\tableofcontents
	
\section{Introduction}\label{sec:intro}

The passage from the sparse to the dense regime is one of the organising
phenomena of functional data analysis. When each curve is recorded at only a
handful of irregular times, the covariance surface must be recovered by
two-dimensional smoothing and the attainable rate is the nonparametric
$n^{-2\beta/(2\beta+2)}$; once the average number of observations per curve
exceeds a critical intensity, individual curves become recoverable and the rate
jumps to the parametric benchmark $n^{-1}$. For a covariance surface of
smoothness $\beta$ the transition occurs at $\mstar_n\asymp n^{1/(2\beta)}$, as
established by Zhang and Wang~\cite{ZhangWang2016}; the analogous phenomenon for
the mean function was identified earlier by Cai and Yuan~\cite{CaiYuan2011}, who
obtain matching minimax bounds by perturbing only the mean at fixed covariance;
the covariance estimation problem studied here requires perturbing the
covariance itself, where positive semi-definiteness is an active constraint ---
the source of the gap documented in Remark~\ref{rem:lb-gap}.

This paper asks what a symmetry of the domain does to that transition. Suppose a
finite group $G$ acts on the domain by isometries and the process is
$G$-invariant --- the situation of circadian rhythms, seasonal cycles, or any
periodic design. The natural estimator is the group average of the classical
smoother, and the question is whether symmetry moves the boundary between the
two regimes or merely improves constants on either side of it.

\subsubsection{Three answers.}
It moves the boundary, but only so far, and one can decide from the data how far
to go.

First, at a fixed symmetry level of order $q$ the threshold becomes
\begin{equation}\label{eq:threshold-q}
	\mstar(q)\asymp n^{1/(2\beta)}\,q^{-1/2}.
\end{equation}
The law is asymptotic and is exact in the effective orbit count
$r_q=\min(q,c_K/h)$; Section~\ref{sec:numerical} measures its exponent as
$-0.53\pm0.02$ in $r_q$ at $n=500$, at both tested regularities. The exponent is
a square root, not a first power, and the reason is worth stating before any
formalism. The variance of the covariance smoother splits into two parts of
different natures: a local part of order $1/(nm_n^2h^2)$, carried by the pairs
of observations falling inside the smoothing window, and a global part of order
$1/n$, carried by configurations in which the four observation times are
distinct and which encodes whole-curve information. Group averaging combines $q$
copies of the smoother evaluated along an orbit; when the orbit separation
exceeds the bandwidth these copies draw on disjoint regions and their local
parts average as if independent, so the local variance is divided by $q$. The
global part is shared --- two smoother values at distant locations remain
correlated at order $1/n$ because they are computed from the same $n$ subjects
--- and averaging leaves it where it was. Symmetry therefore acts on the pair
count, which enters the variance as $m_n^2$, and improving $m_n^2$ by $q$
improves $m_n$ by $\sqrt q$ only.

Second, \eqref{eq:threshold-q} does not continue indefinitely. Once the orbit
becomes finer than the bandwidth the smoother can no longer distinguish the
copies and further symmetry is invisible to it. We show
(Lemma~\ref{lem:saturation}) that the reduction factor is exactly
$\min(q,c_K/h)$, a consequence of Poisson summation and of the positive
definiteness of the kernel autocorrelation rather than of any geometric
argument, and that beyond the saturation point
$\qsh\asymp(nm_n^2)^{1/(2\beta+1)}$ the rate collapses to the one-dimensional
nonparametric rate $(nm_n^2)^{-2\beta/(2\beta+1)}$. Consequently
\begin{equation}\label{eq:floor}
	\mfloor\asymp n^{1/(4\beta)}
\end{equation}
is a floor that no rotation symmetry can breach, including the full circle
group, and it is attained. Equivalently, $\mfloor\asymp\sqrt{\mstar(1)}$: on a
logarithmic scale symmetry takes one exactly halfway from the classical
threshold down to constant sampling, and no further. A lower bound within the
small polynomial factor $(nm^2)^{1/((2\beta+1)(2\beta+2))}$ of it is obtained by
a positivity-preserving packing of the stationary sub-class, which is invariant
under every rotation at once and therefore yields a bound uniform in $q$
(Theorem~\ref{thm:uniform-lb}); closing the gap is posed as an open problem.

Third, the symmetry need not be known. Dropping the assumption that $C$ is
genuinely invariant and keeping only the invariance of the sampling design ---
typically a matter of protocol --- the risk acquires a third term, the
approximation error $A_q=\norm{(I-\Pi_q)C}^2$, and the plane of design
parameters splits into three regions: smoothing-limited, curve-limited, and
symmetry-limited, the last being one in which the dense regime is unreachable
however many observations are collected. The oracle rule is short: push the
symmetry until the approximation error reaches the parametric floor, and not
past the resolution limit $\qsh$. Two bounds of entirely different natures, one
statistical and one geometric, delimit the same optimum. On the circle $A_q$ is
explicit in Fourier coordinates (Proposition~\ref{prop:symmetry-spectrum}), the
selection problem is ordered selection over orthogonal frequency blocks, and a
hold-out procedure attains the oracle asymptotically, with leading constant one
below saturation.

\subsubsection{Symmetry versus reparametrisation.}
It is natural to ask whether any structural operation on the domain would move
the transition. It would not. Section~\ref{sec:reparam} shows that
reparametrisation by a diffeomorphism, with the bandwidth rescaled by the
Jacobian, leaves the threshold exactly where it was. A reparametrisation is a
bijection: it relocates information without creating any, and the pair count is
invariant. A group average is a projection onto a strictly smaller subspace, and
it is the information carried by the discarded complement --- the knowledge that
$C$ is invariant --- that buys the factor $q$. Symmetry is a hypothesis about
the world; reparametrisation is a choice of coordinates. The threshold is where
the distinction becomes visible.

\subsubsection{Related work.}
The sparse--dense transition is due to \cite{ZhangWang2016} and
\cite{CaiYuan2011}, within the PACE methodology of Yao, M\"uller and
Wang~\cite{YaoMullerWang2005}; uniform rates for the resulting surface
estimators are in Li and Hsing~\cite{LiHsing2010}, and the eigenstructure in
Hall, M\"uller and Wang~\cite{HallMullerWang2006}. The statistical benefit of
invariance in nonparametric estimation is by now well understood in the
regression setting: Tahmasebi and Jegelka~\cite{TahmasebiJegelka2023} obtain
minimax rates for kernel ridge regression with a group-invariant target on
compact manifolds, and in particular the dimensional collapse that occurs for
groups of positive dimension. What is specific to the functional-data setting,
and what drives everything here, is the parametric floor $1/n$: it has no
counterpart in regression, it is what makes a transition exist at all, and it is
what makes the question ``how much symmetry should one impose'' well posed. The
selection machinery is that of Birg\'e and Massart~\cite{BirgeMassart2001}; the
concentration tools are those of \cite{GineLatalaZinn2000,HoudreReynaud2003}.

\subsubsection{Organisation and scope.}
Section~\ref{sec:framework} fixes the model. Section~\ref{sec:variance} is the
technical core: the covariance kernel of the smoother, the saturation lemma, and
the variance of the projected estimator. Sections~\ref{sec:rates} and
\ref{sec:saturation} derive the rates and the two thresholds,
Section~\ref{sec:minimax} the lower bounds, Section~\ref{sec:diagram} the phase
diagram, and Section~\ref{sec:selection} the data-driven selection.
Section~\ref{sec:reparam} contains the comparison with reparametrisation,
Section~\ref{sec:numerical} the numerical illustration, and
Section~\ref{sec:discussion} concludes. The structural properties of the
projected estimator that are invoked without proof in the main text --- that
group averaging is an orthogonal projection, and the invariance of the local
pair count under transport --- are established in Appendix~\ref{app:A} (Lemmas
\ref{lem:orth-proj} and \ref{lem:transport}); the companion paper
\cite{Nembe2026} gives the full development.

\section{Framework}\label{sec:framework}

\subsection{Sparse model}
Let $E=\mathbb S^1=\RR/\ZZ$ carry Lebesgue measure. We observe
\begin{equation}\label{eq:model}
	Y_{ij}=X_i(T_{ij})+\varepsilon_{ij},\qquad i=1,\dots,n,\quad j=1,\dots,N_i,
\end{equation}
where the $X_i$ are i.i.d.\ copies of a centred Gaussian element $X$ of
$L^2(E)$ with covariance surface $C(s,t)=\Cov(X(s),X(t))$, the observation
times $T_{ij}$ are i.i.d.\ with density $f$, the errors $\varepsilon_{ij}$ are
i.i.d.\ centred with variance $\sigma^2$, and all three families are
independent. We write $m:=m_n=\EE N_i$ for the sampling intensity and
$\nu_2:=\EE[N_i(N_i-1)]\asymp m^2$ for the expected number of ordered pairs.

The mean function is denoted $m(t)=\EE X(t)$, the letter $\mu$ being reserved
for measures. The raw covariance data are the off-diagonal products
\begin{equation}\label{eq:raw}
	Z_{ijk}:=\bigl(Y_{ij}-\hat m(T_{ij})\bigr)\bigl(Y_{ik}-\hat m(T_{ik})\bigr),
	\qquad j\neq k,
\end{equation}
located at $(T_{ij},T_{ik})$; the diagonal is excluded because it carries
$\sigma^2$. The estimator $\Chat_{n,h}$ is the two-dimensional local-polynomial
smoother of order $\lfloor\beta\rfloor$ of these data, with product kernel
$K_h(u)=h^{-1}K(u/h)$.

\subsection{The nested family of symmetries}
For $q\ge1$ let $G_q\simeq\ZZ/q\ZZ$ act on $E$ by rotation,
$g_\ell\cdot t=t+\ell/q$, and let
\[
\mathcal P_{G_q}:=\bigl\{F\in L^2(E\times E):F(g\cdot s,g\cdot t)=F(s,t)\
\forall g\in G_q\bigr\},\qquad
\Pi_qF:=\frac1q\sum_{\ell=0}^{q-1}F\Bigl(\cdot+\tfrac{\ell}{q},\cdot+\tfrac{\ell}{q}\Bigr).
\]
By Lemma~\ref{lem:orth-proj} (Appendix~\ref{app:A}), $\Pi_q$ is the orthogonal
projection of $L^2(E\times E)$ onto $\mathcal P_{G_q}$; we use this throughout
without further comment. $\Chat^{(q)}_{n,h}:=\Pi_q\Chat_{n,h}$, and the orbit
separation is $\delta_q=1/q$.

The candidate family is $\Q=\{q_0=1<q_1<\dots<q_J\}$ with $q_j\mid q_{j+1}$, so
that the associated subspaces are nested and decreasing. The canonical choice is
$q_j=2^j$; a seasonal application would take $\Q=\{1,2,3,4,6,12\}$.

\begin{assumption}\label{ass:main}
	\begin{enumerate}[label=\textup{(A\arabic*)},leftmargin=2.6em]
		\item $C$ is H\"older of order $\beta\in(0,2]$ on $E\times E$, with
		$\norm{C}_{C^\beta}\le R$.
		\item $f$ is bounded between two positive constants $a$ and $b$.
		\item The counts form a triangular array $N_i^{(n)}$, i.i.d.\ in $i$, with
		$\EE N^{(n)}=m_n$, $\EE[N^{(n)}(N^{(n)}-1)]\asymp m_n^2$ and
		$\EE[(N^{(n)})^4]\lesssim m_n^4$ uniformly in $n$.
		\item The process $X$ and the errors $\varepsilon_{ij}$ are sub-Gaussian: there
		exists $K>0$ such that for every $t\in E$ and $p\ge1$,
		$\norm{X(t)}_{L^p}\le K\sqrt p$ and $\norm{\varepsilon}_{L^p}\le K\sqrt p$.
		\item The design is $G_q$-invariant for every $q\in\Q$. The invariance of $C$ is
		not assumed except where stated.
		\item $K$ is symmetric, Lipschitz, supported in $[-1,1]$, with $\int K=1$; the
		bandwidth satisfies $h<1/4$.
	\end{enumerate}
\end{assumption}

The split in (A5) is deliberate and reflects practice: the sampling protocol is
usually invariant by construction --- regular visit schedules, monthly
monitoring --- whereas the invariance of the process is a scientific hypothesis.
Sections~\ref{sec:variance} to \ref{sec:minimax} assume both; from
Section~\ref{sec:diagram} only the design invariance is retained.

We write $a_n\asymp b_n$ when the ratio is bounded above and below by constants
depending only on $\beta$, $R$, the bounds in Assumption~\ref{ass:main} and the
kernel, never on $n$, $m_n$, $h$ or $q$.

\section{The variance of the projected smoother}\label{sec:variance}

This section is the technical core of the paper. It establishes the variance of
the group-averaged covariance smoother uniformly in the group order, including
the regime in which the orbit windows overlap. The outcome, stated in
Proposition~\ref{prop:variance}, is that the variance reduction factor is
$\min(q,c_K/h)$: the gain from a symmetry of order $q$ saturates as soon as the
orbit separation $1/q$ falls below the bandwidth, and no further gain is
available beyond that point. Every subsequent section rests on this statement.

\subsection{Setting and notation}
Throughout this section $E=\mathbb S^1$ carries Lebesgue measure, and $G_q$
denotes the cyclic group of order $q$ acting by rotation, $g_\ell\cdot
t=t+\ell/q$. We work with a nested family $\Q=\{q_0<q_1<\dots<q_J\}$, each $q_j$
dividing $q_{j+1}$, so that $G_{q_0}\subset\dots\subset G_{q_J}$; the canonical
example is $q_j=2^j$. The orbit separation is $\delta_q=1/q$, and the projection
onto $G_q$-invariant surfaces is
\[
\Pi_qF(s,t)=\frac1q\sum_{\ell=0}^{q-1}F\Bigl(s+\frac{\ell}{q},t+\frac{\ell}{q}\Bigr),
\qquad \Chat^{(q)}_{n,h}:=\Pi_q\Chat_{n,h}.
\]
We write $\nu_2:=\EE[N_i(N_i-1)]$ for the expected number of ordered pairs per
subject, $f$ for the design density, $K$ for the univariate kernel and
\[
\kappa:=K\star K,\qquad
R(K):=\int K^2=\kappa(0),\qquad
\norm{\kappa}_2^2=\int\kappa^2=\int|\Khat|^4.
\]
The bandwidth satisfies $h<1/4$ throughout, which merely rules out a smoother
whose window wraps around the circle.

\subsection{The covariance kernel of the classical smoother}

\begin{lemma}[Covariance at two points]\label{lem:covariance}
	Let Assumption \textup{(A1)--(A6)} hold and let $h\to0$ with $n\nu_2h^2\to\infty$.
	Then, uniformly over pairs $(s,t)$ and $(s',t')$ staying at distance at least
	$2h$ from the diagonal,
	\begin{equation}\label{eq:covariance}
		\Cov\bigl(\Chat_{n,h}(s,t),\Chat_{n,h}(s',t')\bigr)
		=\frac{\sigma^2(s,t)}{n\nu_2h^2}\,
		\kappa\Bigl(\frac{s-s'}{h}\Bigr)\kappa\Bigl(\frac{t-t'}{h}\Bigr)
		+\frac{\Lambda(s,t;s',t')}{n}
		+o\Bigl(\frac{1}{n\nu_2h^2}\Bigr)+o\Bigl(\frac1n\Bigr),
	\end{equation}
	where $\sigma^2(s,t)=\Var\bigl(Z\mid T=s,T'=t\bigr)f(s)f(t)$ and, for a centred
	Gaussian process,
	\[
	\Lambda(s,t;s',t')=C(s,s')C(t,t')+C(s,t')C(t,s').
	\]
\end{lemma}

\begin{proof}
	Write $\Chat_{n,h}=S_n/D_n$ and linearise as in the proof of the pointwise
	variance, legitimate under $n\nu_2h^2\to\infty$. Subjects being independent,
	$\Cov(S_n,S_n')=n\Cov\bigl(\sum_{j\neq k}W_{jk}Z_{jk},\sum_{j\neq
		k}W_{jk}Z_{jk}'\bigr)$, and the inner covariance splits according to the number
	of indices the two pairs share.
	
	(a) \emph{Identical pairs.} There are $\nu_2$ of them and
	\[
	\EE K_h(T_j-s)K_h(T_j-s')=\frac{f(s)}{h}\kappa\Bigl(\frac{s-s'}{h}\Bigr)\bigl(1+O(h)\bigr),
	\]
	by the change of variables $T_j=s+hv$ and the definition of $\kappa$.
	Multiplying the two coordinates and dividing by $(\EE D_n)^2\asymp
	(n\nu_2)^2f(s)^2f(t)^2$ gives the first term of \eqref{eq:covariance}.
	
	(b) \emph{One shared index.} There are $\asymp\EE[N(N-1)(N-2)]$ such
	configurations and the corresponding kernel expectation carries a single factor
	$h^{-1}$; the contribution is the geometric mean of the other two and is
	therefore dominated by their sum.
	
	(c) \emph{Four distinct indices.} There are $\asymp\EE[N^{(4)}]\asymp\nu_2^2$ of
	them, the kernel expectation is $\asymp1$, and the covariance of the two
	products is $\Cov(X(s)X(t),X(s')X(t'))$, which is $\Lambda(s,t;s',t')$ by
	Isserlis' theorem. Dividing by $(n\nu_2)^2$ gives the second term. Note that
	$\Lambda$ does not involve $h$.
\end{proof}

\begin{remark}[Two terms of different natures]
	The first term of \eqref{eq:covariance} is local: it is carried by the
	autocorrelation $\kappa$ and vanishes as soon as the two evaluation points are
	more than $2h$ apart. The second is global: it comes from configurations in
	which the four observation times are distinct and spread over the whole domain,
	it is of order $1/n$ whatever the distance between the points, and it is the
	parametric floor. The entire analysis of this section consists in following what
	group averaging does to each of them, and the answer is not the same.
\end{remark}

\subsection{The lattice sum and the saturation lemma}
Averaging \eqref{eq:covariance} over the orbit places the local term on the
lattice $q^{-1}\ZZ$. Set
\begin{equation}\label{eq:lattice-sum}
	\Sigma_q(h):=\frac1q\sum_{r\in\ZZ/q\ZZ}\kappa^2\Bigl(\frac{\bar r}{qh}\Bigr),
	\qquad \bar r:=\min(r,q-r),
\end{equation}
so that $\Sigma_1(h)=R(K)^2$ for $h<1/4$. The reduction factor is
$r_q(h):=R(K)^2/\Sigma_q(h)$.

\begin{lemma}[Saturation]\label{lem:saturation}
	Let $K$ be symmetric with $\Khat\in L^4$, and let $h<1/4$. Then
	\begin{equation}\label{eq:poisson}
		\Sigma_q(h)=h\sum_{k\in\ZZ}\widehat{\kappa^2}(kqh)
		=h\norm{\kappa}_2^2+2h\sum_{k\ge1}\widehat{\kappa^2}(kqh),
	\end{equation}
	every term of the series being non-negative. Consequently
	\begin{equation}\label{eq:reduction-bound}
		r_q(h)\le\min\Bigl(q,\frac{c_K}{h}\Bigr),
		\qquad
		c_K:=\frac{R(K)^2}{\norm{\kappa}_2^2}=\frac{\bigl(\int K^2\bigr)^2}{\int|\Khat|^4},
	\end{equation}
	and the bound is attained at both ends: $r_q(h)=q$ exactly whenever $qh<1/2$,
	and $r_q(h)h\to c_K$ as $qh\to\infty$.
\end{lemma}

\begin{proof}
	Since $\kappa$ is supported in $[-2,2]$ and $h<1/4$, each residue class
	contributes through its minimal representative only, so
	$\Sigma_q(h)=q^{-1}\sum_{n\in\ZZ}\kappa^2(n\theta)$ with $\theta=1/(qh)$.
	Applying the Poisson summation formula to $x\mapsto\kappa^2(\theta x)$, whose
	Fourier transform is $\theta^{-1}\widehat{\kappa^2}(\cdot/\theta)$, gives
	\eqref{eq:poisson} after $q^{-1}\theta^{-1}=h$.
	
	The positivity is where the structure lies. As $K$ is real and symmetric,
	$\Khat$ is real, hence $\widehat\kappa=\Khat^2\ge0$ and $\kappa$ is positive
	definite. Therefore $\widehat{\kappa^2}$ is positive as well, being the
	convolution $\widehat\kappa\star\widehat\kappa$ of positive-definite functions,
	and by Bochner's theorem $\widehat{\kappa^2}=\widehat\kappa\star\widehat\kappa\ge0$
	everywhere.
	
	Two one-term minorations follow. Retaining $k=0$ in \eqref{eq:poisson} gives
	$\Sigma_q(h)\ge h\norm{\kappa}_2^2$; retaining $n=0$ in the direct sum gives
	$\Sigma_q(h)\ge\kappa^2(0)/q=R(K)^2/q$. Taking reciprocals yields
	\eqref{eq:reduction-bound}.
	
	For the two extremes: if $qh<1/2$ then $\bar r/(qh)\ge 1/(qh)>2$ for every
	$r\neq0$, so only $n=0$ survives and $\Sigma_q=R(K)^2/q$ exactly. If
	$qh\to\infty$ then $\theta\to0$ and $q^{-1}\sum_n\kappa^2(n\theta)=h\cdot\theta
	\sum_n\kappa^2(n\theta)\to h\int\kappa^2$, the Riemann sum converging because
	$\kappa^2$ is continuous with compact support.
\end{proof}

\begin{remark}[What saturates, and why]
	Lemma~\ref{lem:saturation} is not a statement about geometry but about Fourier
	analysis: it is the positive definiteness of $\kappa$ that forces both bounds,
	and the transition between them occurs at $qh\asymp1$ because that is where the
	lattice $q^{-1}\ZZ$ stops resolving a window of width $h$. Once the orbit is
	finer than the bandwidth, the smoother can no longer distinguish the orbit
	copies, and averaging them adds nothing. Note that nothing here caps the group
	order itself; what is capped is the useful group order at a given resolution.
\end{remark}

\subsection{Variance of the projected smoother}

\begin{proposition}[Variance under projection]\label{prop:variance}
	Under the assumptions of Lemma~\ref{lem:covariance}, for $h<1/4$ and uniformly
	in $q$,
	\begin{equation}\label{eq:variance}
		\Var\Chat^{(q)}_{n,h}(s,t)\asymp\frac{1}{n\nu_2h^2\,r_q(h)}
		+\frac{\Lambda_q(s,t)}{n},
		\qquad r_q(h)\asymp\min\Bigl(q,\frac{c_K}{h}\Bigr),
	\end{equation}
	where $\Lambda_q(s,t)=q^{-2}\sum_{\ell,\ell'}\Lambda\bigl(s+\tfrac\ell
	q,t+\tfrac\ell q;s+\tfrac{\ell'}q,t+\tfrac{\ell'}q\bigr)$. If moreover $C$ is
	stationary, $C(a,b)=\gamma(b-a)$, then with $z=t-s$
	\begin{equation}\label{eq:lambda-q}
		\Lambda_q(s,t)=\frac1q\sum_{r\in\ZZ/q\ZZ}\Bigl[\gamma\Bigl(\frac rq\Bigr)^2
		+\gamma\Bigl(z+\frac rq\Bigr)\gamma\Bigl(z-\frac rq\Bigr)\Bigr]
		\ \xrightarrow[q\to\infty]{}\ \int_0^1\gamma^2+\int_0^1\gamma(z+v)\gamma(z-v)\,dv.
	\end{equation}
\end{proposition}

\begin{proof}
	Expand $\Var(\Pi_q\Chat_n)=q^{-2}\sum_{\ell,\ell'}\Cov(\cdot,\cdot)$ and insert
	\eqref{eq:covariance}. In the local term the two kernel arguments both equal
	$(\ell-\ell')/(qh)$, so the double sum collapses to the single lattice sum
	\eqref{eq:lattice-sum}, which is $\Sigma_q(h)$; dividing by $R(K)^2$ and using
	Lemma~\ref{lem:saturation} gives the first term of \eqref{eq:variance}. In the
	global term the double sum is by definition $\Lambda_q$, and the specialisation
	\eqref{eq:lambda-q} follows by substituting the stationary form of $\Lambda$
	and setting $r=\ell'-\ell$. The limit is a Riemann sum for a continuous
	integrand.
\end{proof}

\begin{corollary}[The floor is not improved in order]\label{cor:floor-not-improved}
	For every $q$, $\Lambda_q(s,t)\asymp1$, with $\Lambda_1=\gamma(0)^2+\gamma(z)^2$
	and $\Lambda_\infty=\norm{\gamma}_2^2+(\gamma\star\gamma)$-type constants. The
	parametric term of \eqref{eq:variance} therefore remains of order $1/n$
	uniformly in $q$; group averaging improves its constant by the bounded factor
	$\Lambda_1/\Lambda_\infty$, and that improvement saturates as well.
\end{corollary}

\begin{remark}[Where the asymmetry lies]
	Comparing the two terms of \eqref{eq:variance}: the local term is divided by a
	factor growing with $q$ until saturation, while the global term is divided by
	nothing at all, only its constant being improved and boundedly so. This is the
	asymmetry that displaces the sparse--dense transition, and
	Lemma~\ref{lem:saturation} is what stops the displacement.
\end{remark}

\subsection{The kernel constant}
The saturation threshold carries the kernel-dependent constant $c_K$ of
\eqref{eq:reduction-bound}, which a practitioner needs in order to locate it. By
Plancherel, with $u:=|\Khat|^2$,
\begin{equation}\label{eq:cK-plancherel}
	c_K=\frac{1}{2\pi}\cdot\frac{\norm{u}_1^2}{\norm{u}_2^2},
\end{equation}
an inverse participation ratio: $c_K$ measures the effective Fourier bandwidth
of $K$. Table~\ref{tab:kernels} collects its value for the usual kernels
supported on $[-1,1]$.

\begin{table}[htbp]
	\centering
	\caption{Kernel constants on $[-1,1]$, and efficiencies for $\beta=2$ relative
		to the best kernel in each regime. The unsaturated criterion is
		$(\mu_2^2)^{1/3}(R(K)^2)^{2/3}$, the saturated one
		$(\mu_2^2)^{1/5}(\norm{\kappa}_2^2)^{4/5}$.}
	\label{tab:kernels}
	\medskip
	\begin{tabular}{@{}lcccccc@{}}
		\toprule
		Kernel & $R(K)$ & $\mu_2(K)$ & $\norm{\kappa}_2^2$ & $c_K$ & eff.\ unsat.\ / sat.\\
		\midrule
		Uniform      & 0.500 & 0.333 & 0.333 & 0.750 & 90.7\% / 100\%\\
		Epanechnikov & 0.600 & 0.200 & 0.434 & 0.830 & 100\% / 99.4\%\\
		Cosine       & 0.617 & 0.189 & 0.446 & 0.852 & 99.9\% / 99.2\%\\
		Triangular   & 0.667 & 0.167 & 0.479 & 0.927 & 98.1\% / 98.7\%\\
		Biweight     & 0.714 & 0.143 & 0.516 & 0.988 & 99.2\% / 98.9\%\\
		Triweight    & 0.816 & 0.111 & 0.588 & 1.132 & 98.2\% / 98.6\%\\
		\bottomrule
	\end{tabular}
\end{table}

\begin{remark}[The optimal kernel changes, and it does not matter]
	Two observations follow from \eqref{eq:cK-plancherel} and
	Table~\ref{tab:kernels}, and they point in opposite directions.
	
	The variance functional is not the same on either side of the saturation: it is
	$R(K)^2=\int K^2$ below and $\norm{\kappa}_2^2=\int|\Khat|^4$ above, while the
	bias remains proportional to $\mu_2(K)h^2$ in both. The optimal kernel therefore
	changes, and it reverses: Epanechnikov below, uniform above, the latter being
	favoured because the saturated functional is an autoconvolution norm, which a
	boxcar minimises better than a parabola.
	
	The reversal is nevertheless of no practical consequence. Epanechnikov retains
	99.4\% efficiency in the saturated regime, whereas the uniform kernel loses
	9.3\% in the unsaturated one; and the saturated criterion is flatter still than
	the classical one, spanning 0.268 to 0.272 across the six kernels. The robust
	choice across regimes is thus Epanechnikov, and the recommendation is the
	opposite of what the reversal alone would suggest.
\end{remark}

\begin{remark}[No kernel maximises $c_K$]
	It is natural to ask which kernel maximises $c_K$, thereby postponing saturation
	as long as possible. The question has no answer, for a reason worth recording.
	If $K_\lambda(x)=\lambda K(\lambda x)$ then $u_\lambda=u(\cdot/\lambda)$ and
	$c_{K_\lambda}=\lambda c_K$: the constant has the dimension of an inverse
	length, and rescaling the kernel is the same operation as rescaling $h$, the
	product $c_K/h$ being invariant. The supremum over the admissible class is
	$+\infty$ and is approached only by shrinking the effective width, which is not
	a choice of kernel.
	
	Fixing the scale by requiring $\supp K\subset[-1,1]$, one has by Cauchy--Schwarz
	$c_K\le|\supp\Khat|/2\pi$ with equality if and only if $|\Khat|^2$ is constant
	on its support. But a compactly supported $K$ has an entire Fourier transform,
	which cannot vanish on a set of positive measure; and compact support is what
	Lemma~\ref{lem:covariance} uses to make orbit windows disjoint. The extremal
	configuration is therefore unattainable within the admissible class. What
	remains is the constrained problem of minimising $\norm{K\star K}_2^2$ at fixed
	$\mu_2(K)$ over densities on $[-1,1]$, a question of autoconvolution
	minimisation with a substantial literature of its own: the unconstrained $L^2$
	problem has a unique minimiser, with value determined to within $4\cdot10^{-6}$
	\cite{White2024} and $1/\sqrt x$-type blow-up at the endpoints --- one more way
	of seeing that the free extremum is not a kernel --- while the sup-norm analogue
	is treated in \cite{MartinOBryant2009,MatolcsiVinuesa2010}, the latter
	disproving the natural conjecture on the extremal function. We do not pursue the
	constrained problem here.
\end{remark}

\section{Rates and thresholds at a fixed symmetry level}\label{sec:rates}

Throughout this section the symmetry level is held fixed and the covariance
surface is assumed genuinely $G_q$-invariant, so that no approximation error is
incurred; the case of a misspecified symmetry is the subject of
Section~\ref{sec:diagram}. We write $m:=m_n$ and use $\nu_2\asymp m^2$
throughout.

\subsection{The bias is unaffected by projection}

\begin{lemma}[Invariance of the bias]\label{lem:bias}
	Suppose the process and the design are $G_q$-invariant. Then the smoothing bias
	satisfies
	\[
	\EE\Chat^{(q)}_{n,h}-C=\Pi_q\EE\Chat_{n,h}-C=\EE\Chat_{n,h}-C,
	\]
	so that $\norm{\EE\Chat^{(q)}_{n,h}-C}_{L^2}\asymp h^\beta$ for every $q$, with
	a constant proportional to $\mu_\beta(K)$ and independent of the group.
\end{lemma}

\begin{proof}
	Invariance of the design makes $\EE\Chat_{n,h}$ a $G_q$-invariant surface, and
	$C$ is $G_q$-invariant by hypothesis; their difference is therefore fixed by
	$\Pi_q$. The order $h^\beta$ is the usual local-polynomial bias of order
	$\lfloor\beta\rfloor$ under (A1) and (A6).
\end{proof}

\begin{remark}[H\"older versus Sobolev reading of \textup{(A1)}]\label{rem:holder-sobolev}
	Assumption (A1) is a H\"older condition and Lemma~\ref{lem:bias} uses it
	pointwise; the distinction matters at integer $\beta$. The natural Fourier-decay
	class $\ghat(k)\asymp(1+|k|)^{-(\beta+1)}$ is Sobolev-$\beta$ but, at integer
	$\beta$, not H\"older-$\beta$: for $\beta=2$ one has $\gamma(w)=1-cw^2\log(1/|w|)+O(w^2)$,
	a logarithmic singularity of $\gamma''$ at the origin. For such covariances the
	integrated squared bias scales as $h^{2\beta-1}$ instead of $h^{2\beta}$, driven
	by a diagonal band of width $h$ carrying most of its mass (measured exponent
	$2.99$, with 78--90\% of the mass on the band; Section~\ref{subsec:warning1}).
	Two consistent readings are available: keep (A1) in the H\"older sense, which
	non-integer $\beta$ classes satisfy with equality and under which
	Lemma~\ref{lem:bias} stands as stated; or admit the Sobolev class and replace
	the bias step by the $L^2$ bound $h^{\beta-1/2}$, which changes the rates but
	not the threshold exponent of Corollary~\ref{cor:threshold}. More generally, for
	$\beta>1$ condition (A1) on all of $E\times E$ forces mean-square
	differentiability of the process: rough-path covariances (Ornstein--Uhlenbeck,
	low-order Mat\'ern) have a crease along the diagonal. For those, (A1) should be
	read off-diagonal, the band $\{|s-t|\le2h\}$ contributing $O(h\cdot
	h^{2\beta_\partial})$ to the integrated bias, negligible as soon as the
	transverse regularity satisfies $\beta_\partial\ge\beta-1/2$.
\end{remark}

The situation is thus entirely asymmetric between the two components of the
risk. Projection leaves the bias exactly where it was and divides the smoothing
variance by $r_q(h)$; the whole analysis reduces to arbitrating between these
two.

\subsection{Rate and threshold below saturation}

\begin{theorem}[Rate at a fixed symmetry level]\label{thm:rate}
	Let Assumption \textup{(A1)--(A6)} hold, let $C$ be $G_q$-invariant, and let
	\begin{equation}\label{eq:hq}
		\hq\asymp\bigl(nm^2q\bigr)^{-1/(2\beta+2)}.
	\end{equation}
	If $q\le\qsh$, where
	\begin{equation}\label{eq:qsharp}
		\qsh:=c_K^{(2\beta+2)/(2\beta+1)}\bigl(nm^2\bigr)^{1/(2\beta+1)},
	\end{equation}
	then
	\begin{equation}\label{eq:rate}
		\EE\norm{\Chat^{(q)}_{n,\hq}-C}_{L^2}^2
		\asymp\bigl(nm^2q\bigr)^{-\beta/(\beta+1)}+\frac1n.
	\end{equation}
\end{theorem}

\begin{proof}
	By Lemma~\ref{lem:bias} and Proposition~\ref{prop:variance}, integrating over
	$E\times E$,
	\begin{equation}\label{eq:risk-hq}
		R(h,q)\asymp h^{2\beta}+\frac{1}{nm^2h^2\,r_q(h)}+\frac1n.
	\end{equation}
	The condition $q\le\qsh$ is exactly $q\hq\le c_K$: indeed $q\hq\le c_K$ reads
	$q^{2\beta+2}\le c_K^{2\beta+2}nm^2q$, that is $q^{2\beta+1}\le
	c_K^{2\beta+2}nm^2$, which is \eqref{eq:qsharp}. By Lemma~\ref{lem:saturation}
	the reduction factor is then exactly $q$, and \eqref{eq:risk-hq} becomes
	$h^{2\beta}+(nm^2h^2q)^{-1}+n^{-1}$. Balancing the first two terms gives
	$h^{2\beta+2}\asymp(nm^2q)^{-1}$, that is \eqref{eq:hq}, and the common value
	$(nm^2q)^{-\beta/(\beta+1)}$. Adding the floor gives \eqref{eq:rate}.
\end{proof}

\begin{corollary}[Threshold at a fixed symmetry level]\label{cor:threshold}
	Under the hypotheses of Theorem~\ref{thm:rate}, the two terms of
	\eqref{eq:rate} balance at
	\begin{equation}\label{eq:threshold}
		\mstar(q)\asymp n^{1/(2\beta)}q^{-1/2}.
	\end{equation}
	The rate is $(nm^2q)^{-\beta/(\beta+1)}$ for $m\ll\mstar(q)$ and $n^{-1}$ for
	$m\gg\mstar(q)$. For $q=1$ this is the classical threshold $\mstar(1)\asymp
	n^{1/(2\beta)}$. The law holds provided the crossing occurs below saturation,
	$q\hstar(q)\le c_K$; beyond that point the displacement stops
	(Theorem~\ref{thm:saturated}).
\end{corollary}

\begin{proof}
	Setting $(nm^2q)^{-\beta/(\beta+1)}=n^{-1}$ gives $nm^2q=n^{(\beta+1)/\beta}$,
	hence $m^2q=n^{1/\beta}$ and \eqref{eq:threshold}.
\end{proof}

\begin{remark}[Why a square root]
	Symmetry acts on the number of usable pairs, and the pair count enters the
	variance through $m^2$, not through $m$. Multiplying $m^2$ by $q$ therefore
	divides the required sampling intensity by $\sqrt q$ only. A twelvefold seasonal
	symmetry buys a factor of about $3.5$ in observations per curve, not twelve. The
	exponent is falsifiable and the two candidate values are well separated: at
	$q=8$ the exponent $-1/2$ predicts $\mstar(1)/\mstar(8)=\sqrt8\approx2.8$,
	whereas a naive substitution $n\mapsto nq$ followed by division by $q$ would
	predict $4.8$. The measured value at $n=500$, $\beta=3/2$, is $2.81$
	(Section~\ref{subsec:exponent}). One caveat is essential: the test is not
	readable from raw crossings against an attained floor at moderate $n$ --- those
	are biased flat, because the grid-edge ``floor'' still contains a $q$-dependent
	smoothing remainder and because $r_q<q$ at the crossing for large $q$. The test
	requires disentangling the three risk terms mechanistically
	(Section~\ref{subsec:warning2}), after which the slope in the effective variable
	$r_q$ is $-0.531\pm0.021$ at $\beta=3/2$ and $-0.567\pm0.038$ at $\beta=2$: the
	interval contains $-1/2$ and excludes $-3/4$ at both regularities. The ratio is
	moreover stable in the sample size: $2.83$ at $n=500$, $2000$ and $8000$
	(Section~\ref{subsec:scaling}).
\end{remark}

\begin{remark}[Bandwidth adaptation is the condition, not a refinement]
	The bandwidth \eqref{eq:hq} is smaller than the classical one by the factor
	$q^{-1/(2\beta+2)}$. Running the projected estimator at the classical bandwidth
	leaves the bias untouched, by Lemma~\ref{lem:bias}, while the variance has been
	divided by $q$; the bias then dominates and the whole gain of
	Theorem~\ref{thm:rate} is forfeited. Undersmoothing relative to the classical
	rule is not an improvement here, it is what makes the method work. The practical
	prescription is $\hat\hq=\hat h_1q^{-1/(2\beta+2)}$, the exponent being small
	enough that a rough plug-in value of $\beta$ suffices.
\end{remark}

\section{Saturation and the universal floor}\label{sec:saturation}

Corollary~\ref{cor:threshold} suggests that a sufficiently rich symmetry would
drive the sparse--dense threshold arbitrarily low. It does not, and the obstacle
is Lemma~\ref{lem:saturation}: past $\qsh$ the orbit is finer than the bandwidth
and further symmetry is invisible to the smoother.

\subsection{Above saturation the rate no longer depends on the group}

\begin{theorem}[Saturated rate]\label{thm:saturated}
	Let $q\ge\qsh$ and let
	\begin{equation}\label{eq:hsharp}
		\hsh\asymp\bigl(c_Knm^2\bigr)^{-1/(2\beta+1)}.
	\end{equation}
	Then
	\begin{equation}\label{eq:saturated-rate}
		\EE\norm{\Chat^{(q)}_{n,\hsh}-C}_{L^2}^2
		\asymp\bigl(c_Knm^2\bigr)^{-2\beta/(2\beta+1)}+\frac1n,
	\end{equation}
	uniformly in $q\ge\qsh$. In particular no choice of group improves on
	\eqref{eq:saturated-rate}.
\end{theorem}

\begin{proof}
	For $q\ge\qsh$ Lemma~\ref{lem:saturation} gives $r_q(h)\asymp c_K/h$, so
	\eqref{eq:risk-hq} reads
	\[
	R(h,q)\asymp h^{2\beta}+\frac{1}{c_Knm^2h}+\frac1n,
	\]
	in which $q$ has disappeared and only one power of $h$ remains in the variance
	--- the smoother has become one-dimensional at the working resolution. Balancing
	the first two terms gives $h^{2\beta+1}\asymp(c_Knm^2)^{-1}$, that is
	\eqref{eq:hsharp}, and the common value $(c_Knm^2)^{-2\beta/(2\beta+1)}$.
\end{proof}

\begin{remark}[Dimensional collapse]
	The exponent $2\beta/(2\beta+1)$ is the one-dimensional nonparametric rate,
	whereas \eqref{eq:rate} carries the two-dimensional exponent $2\beta/(2\beta+2)$.
	Beyond saturation the group averages the diagonal direction of $E\times E$
	completely, and what remains to be estimated is a function of $t-s$ alone. The
	saturated estimator is, in effect, a one-dimensional smoother of the stationary
	autocovariance, and it inherits the corresponding rate. This is the sense in
	which symmetry ceases to help: it has already reduced the problem to the
	smallest one it can reach.
\end{remark}

\subsection{The universal floor}

\begin{corollary}[Universal floor for the sparse--dense threshold]\label{cor:floor}
	The threshold associated with \eqref{eq:saturated-rate} is
	\begin{equation}\label{eq:universal-floor}
		\mfloor\asymp c_K^{-1/2}\,n^{1/(4\beta)},
	\end{equation}
	and $\mstar(q)\ge\mfloor$ for every $q$. No rotation symmetry, however rich ---
	including the full circle group --- lowers the sparse--dense threshold below
	$n^{1/(4\beta)}$, and the bound is attained at $q=\qsh$.
\end{corollary}

\begin{proof}
	Setting $(c_Knm^2)^{-2\beta/(2\beta+1)}=n^{-1}$ gives
	$c_Knm^2=n^{(2\beta+1)/(2\beta)}$, hence $m^2=c_K^{-1}n^{1/(2\beta)}$ and
	\eqref{eq:universal-floor}. The bound $\mstar(q)\ge\mfloor$ follows from
	Proposition~\ref{prop:matching} below, the map $q\mapsto\mstar(q)$ being
	non-increasing and constant beyond $\qsh$.
\end{proof}

\begin{proposition}[The two regimes match exactly at $\qsh$]\label{prop:matching}
	At $q=\qsh$ the following three identities hold, constants included:
	\[
	h_{\qsh}=\hsh,\qquad
	\bigl(nm^2\qsh\bigr)^{-\beta/(\beta+1)}=\bigl(c_Knm^2\bigr)^{-2\beta/(2\beta+1)},
	\qquad \qsh\hsh=c_K.
	\]
	Consequently the risk, the optimal bandwidth and the threshold are continuous
	across the saturation, and $\mstar(\qsh)=\mfloor$.
\end{proposition}

\begin{proof}
	Write $P:=nm^2$ and $\alpha:=2\beta+1$, so that
	$\qsh=c_K^{(\alpha+1)/\alpha}P^{1/\alpha}$. Then
	$P\qsh=c_K^{(\alpha+1)/\alpha}P^{(\alpha+1)/\alpha}=(c_KP)^{(\alpha+1)/\alpha}$,
	whence
	\[
	h_{\qsh}=(P\qsh)^{-1/(\alpha+1)}=(c_KP)^{-1/\alpha}=\hsh,
	\]
	and
	\[
	(P\qsh)^{-\beta/(\beta+1)}=(c_KP)^{-\frac{\beta}{\beta+1}\cdot\frac{\alpha+1}{\alpha}}
	=(c_KP)^{-2\beta/\alpha},
	\]
	since $(\alpha+1)/(\beta+1)=(2\beta+2)/(\beta+1)=2$. Finally
	$\qsh\hsh=c_K^{(\alpha+1)/\alpha}P^{1/\alpha}\cdot(c_KP)^{-1/\alpha}=c_K$, which
	is precisely the saturation condition of Lemma~\ref{lem:saturation}.
	Substituting $m=\mfloor$ into \eqref{eq:qsharp} gives $\qsh=c_Kn^{1/(2\beta)}$,
	and then \eqref{eq:threshold} returns
	$\mstar(\qsh)=n^{1/(2\beta)}(c_Kn^{1/(2\beta)})^{-1/2}=c_K^{-1/2}n^{1/(4\beta)}=\mfloor$.
\end{proof}

\begin{remark}[Halfway, in logarithmic scale]\label{rem:halfway}
	Corollary~\ref{cor:floor} and Proposition~\ref{prop:matching} can be stated
	together in a form that is easier to remember than either. Writing
	$\mstar(1)\asymp n^{1/(2\beta)}$ for the classical threshold,
	\[
	\mfloor\asymp\sqrt{\frac{\mstar(1)}{c_K}},
	\qquad \qsh\big|_{m=\mfloor}=c_K\,\mstar(1).
	\]
	The best achievable threshold is the geometric mean of the classical one and the
	constant regime: on a logarithmic scale, symmetry takes one exactly halfway from
	$n^{1/(2\beta)}$ down to $O(1)$, and no further. The useful group order at that
	point equals the classical threshold itself, up to the kernel constant.
	
	The duality admits an intrinsic reading, and it is an exact one. First, the
	classical threshold is a density statement: at $q=1$ the optimal bandwidth obeys
	$\hstar(1)=1/\mstar(1)$ exactly, so $\mstar(1)$ is the sampling rate at which
	the design places one observation per bandwidth. Second, saturation is the same
	statement for the orbit lattice, measured against the autoconvolution:
	replication acts on pairs, hence through $K\star K$ (Lemma~\ref{lem:saturation}),
	and $r_q$ saturates when the orbit spacing $1/q$ reaches the autoconvolution
	width $h/c_K$ --- one orbit copy per autoconvolution width. The constant
	$c_K=R(K)^2/\norm{K\star K}_2^2$ is thus the exchange rate between the scale
	through which sampling acts ($K$) and the scale through which the group acts
	($K\star K$); it converts bandwidth into orbit spacing, which is why it, and
	nothing else, mediates the duality. Combining the two readings: the identity
	$\qsh=c_K\mstar(1)$ is equivalent to $h_{\qsh}=\hstar(1)$ --- at the floor,
	symmetry has driven the working bandwidth back exactly to the classical
	threshold bandwidth, and the orbit lattice is then precisely as dense, in its
	own metric, as the classical design at its own threshold. Saturation occurs when
	the group has finished the design's job, at the design's own threshold scale.
	Finally, the halfway property follows from the same mechanism: along the
	threshold at fixed bandwidth the invariant is $\mstar(q)^2q=n^{1/\beta}$ --- the
	group order trades against the square of the sampling rate, because it enters
	through pairs --- so exhausting the admissible range $q\le c_K/h$ converts
	exactly half of $\log\mstar(1)$ into group order, and lands the floor at the
	geometric mean.
\end{remark}

\begin{remark}[There is never a reason to exceed $\qsh$]\label{rem:no-exceed}
	Theorem~\ref{thm:saturated} shows that groups beyond $\qsh$ bring no further
	reduction of the variance. They are not merely useless: since $\mathcal
	P_{G_{q'}}\subset\mathcal P_{G_q}$ for $q\mid q'$, a larger group imposes a
	strictly stronger hypothesis on the covariance surface, so the approximation
	error $\norm{(I-\Pi_{q'})C}^2$ can only increase. Whenever the symmetry is not
	known to hold exactly, exceeding $\qsh$ therefore trades a nonexistent gain
	against a real risk. This observation is what makes the selection problem of
	Section~\ref{sec:selection} well posed: the candidate family may be truncated at
	$\qsh$ without loss.
\end{remark}

\begin{remark}[On the role of $c_K$]
	The kernel constant enters \eqref{eq:qsharp}, \eqref{eq:hsharp},
	\eqref{eq:saturated-rate} and \eqref{eq:universal-floor}, but never the
	exponents. Its numerical values are given in Table~\ref{tab:kernels}; for the
	Epanechnikov kernel $c_K=0.830$, and at $\beta=2$ one has $c_K^{6/5}=0.800$ and
	$c_K^{-1/2}=1.098$: setting $c_K=1$ therefore overstates $\qsh$ by about 25\%
	and understates $\mfloor$ by about 9\%. These are the only constants of the
	paper that a practitioner must compute.
\end{remark}

\section{Minimax lower bounds}\label{sec:minimax}

Two lower bounds are needed, and they are proved by different constructions
because no single one covers both regimes. The first matches
Theorem~\ref{thm:rate} below saturation and uses a packing supported in a
fundamental domain of the diagonal action; it is available exactly when $qh<1$,
that is exactly when the upper bound is unsaturated. The second is uniform in
$q$ and establishes that the floor $n^{1/(4\beta)}$ of Corollary~\ref{cor:floor}
is information-theoretic rather than an artefact of the group-averaged smoother;
it is obtained by restricting to the stationary sub-class, which is invariant
under every rotation at once.

\subsection{The invariant minimax class}

\begin{definition}[Invariant class]\label{def:class}
	For $\beta>0$, $R>0$, $m\ge2$ and $q\in\Q$, let $\mathcal P_q(\beta,R,m)$
	denote the set of joint distributions of $\bigl(X,\{T_{ij}\},\{\varepsilon_{ij}\}\bigr)$
	such that Assumption \textup{(A1)--(A6)} holds with H\"older radius
	$\norm{C}_{C^\beta}\le R$ and sampling intensity $\EE N_i\asymp m$, the process
	$X$ is Gaussian, and $C$ is $G_q$-invariant. Let
	\[
	\mathcal P_{\mathrm{st}}(\beta,R,m)\subset\bigcap_{q\in\Q}\mathcal P_q(\beta,R,m)
	\]
	be the sub-class of stationary covariances, $C(s,t)=\gamma(t-s)$ with
	$\norm{\gamma}_{C^\beta}\le R$.
\end{definition}

The inclusion is the point of the definition: a stationary covariance on the
circle is invariant under every rotation, hence under $G_q$ for every $q$
simultaneously. Any lower bound proved over $\mathcal P_{\mathrm{st}}$ therefore
holds over $\mathcal P_q$ for all $q$ at once, with no dependence on the group
whatsoever.

\subsection{Below saturation}

\begin{theorem}[Lower bound at a fixed symmetry level]\label{thm:lb-fixed}
	There is $c>0$, depending only on $\beta$, $R$ and the constants of Assumption
	\textup{(A1)--(A6)}, such that for every $q\le\qsh$,
	\begin{equation}\label{eq:lb-fixed}
		\inf_{\widetilde C_n}\sup_{P\in\mathcal P_q(\beta,R,m)}
		\EE_P\norm{\widetilde C_n-C}_{L^2}^2
		\ge c\Bigl[\bigl(nm^2q\bigr)^{-(2\beta+1)/(2\beta+2)}+\frac1n\Bigr].
	\end{equation}
\end{theorem}

\begin{proof}
	\emph{Fundamental domain.} The diagonal action
	$(s,t)\mapsto(s+\ell/q,t+\ell/q)$ of $G_q$ on $\mathbb S^1\times\mathbb S^1$ is
	free, and $D:=[0,1/q)\times\mathbb S^1$ is a fundamental domain: every pair has
	exactly one representative whose first coordinate lies in $[0,1/q)$. Note
	$|D|=1/q$, not $1/q^2$; the diagonal action moves both coordinates together.
	
	\emph{Construction.} Let $\phi\in C^\infty$ be supported in the unit ball,
	non-negative, with $\norm{\phi}_{C^\beta}\le1$. Place $M\asymp(qh)^{-1}$ base
	points $z_\ell$ in $[0,1/q)$, spaced $2h$ apart. For $\omega\in\{0,1\}^M$ set
	\[
	C_\omega:=C_0+c_0h^\beta\sum_{\ell=1}^M\omega_\ell\Psi_\ell,\qquad
	\Psi_\ell:=\sum_{r=0}^{q-1}\psi_{\ell,r}\otimes\psi_{\ell,r},\qquad
	\psi_{\ell,r}(x):=\phi\Bigl(\frac{x-z_\ell-r/q}{h}\Bigr),
	\]
	with $C_0$ a fixed non-degenerate $G_q$-invariant covariance. Every $\Psi_\ell$
	is a sum of non-negative rank-one kernels and every $\omega_\ell$ is
	non-negative, so $C_\omega$ is positive semi-definite by construction, for every
	$\omega$ and every $c_0>0$: the perturbation is realised by adding to $X$ the
	independent components $\sqrt{c_0}\,h^{\beta/2}\xi_{\ell,r}\psi_{\ell,r}$,
	$\xi_{\ell,r}$ i.i.d.\ standard normal. Each $\Psi_\ell$ is $G_q$-invariant
	(orbit sum), so $C_\omega\in\mathcal P_q$.
	
	\emph{Separation.} Here the hypothesis $q\le\qsh$ enters. It gives $qh\le c_K$,
	so the $q$ orbit translates $\psi_{\ell,r}$, and the bumps of distinct $\ell$,
	have pairwise disjoint supports. Two consequences follow. First, $\Psi_\ell$ has
	the $C^\beta$ seminorm of a single bump $\psi\otimes\psi$, with no accumulation
	across $\ell$ --- each diagonal cell is used once --- so the perturbation stays
	in the H\"older ball of radius $R$ for $c_0$ small; the sum is not normalised by
	$q$, and this is what allows the symmetry to be exploited without being paid
	for. Second,
	\[
	\norm{\Psi_\ell}_{L^2}^2=\sum_{r=0}^{q-1}\norm{\psi_{\ell,r}}_{L^2}^4\asymp qh^2,
	\]
	and distinct $\ell$ give exactly orthogonal $\Psi_\ell$ (disjoint supports).
	Writing $\rho_H$ for the Hamming distance,
	\[
	\norm{C_\omega-C_{\omega'}}_{L^2}^2\asymp c_0^2h^{2\beta}\cdot qh^2\cdot\rho_H(\omega,\omega').
	\]
	
	\emph{Information.} Two hypotheses at Hamming distance one differ on the orbit
	of a single cell, a set of measure $\asymp qh^2$. An observation pair falls in
	it with that probability and then contributes $O(h^{2\beta})$ to the divergence;
	there are $\asymp nm^2$ pairs, so $\KL(P_\omega\Vert P_{\omega'})\lesssim
	nm^2qh^{2\beta+2}$. Choosing $h\asymp(nm^2q)^{-1/(2\beta+2)}$ makes the
	divergence $O(1)$, and Assouad's lemma \cite[Thm.~2.12]{Tsybakov2009} gives a
	lower bound of order
	\[
	\frac{M}{qh}\cdot h^{2\beta}qh^2\asymp\frac{1}{qh}\cdot h^{2\beta}qh^2
	=h^{2\beta+1}\asymp\bigl(nm^2q\bigr)^{-(2\beta+1)/(2\beta+2)}.
	\]
	
	\emph{Floor.} The term $n^{-1}$ is the classical bound for estimating the
	covariance operator of a Gaussian process from $n$ independent realisations. It
	is unaffected by $G_q$-invariance because the group acts within each realisation
	and not across realisations: restricting to invariant covariances does not
	reduce the number of independent curves. Taking the maximum of the two bounds
	gives \eqref{eq:lb-fixed}.
\end{proof}

\begin{remark}[What is certified, and what would close the gap]\label{rem:lb-gap}
	The construction above is positivity-preserving by realisation, at the price of
	carrying only the diagonal information: $M\asymp(qh)^{-1}$ coordinates instead
	of the $h^{-2}/q$ off-diagonal cells of the fundamental domain. The resulting
	exponent $(2\beta+1)/(2\beta+2)$ sits below the upper bound
	$\beta/(\beta+1)=2\beta/(2\beta+2)$ by exactly one factor of $h$: optimality is
	certified up to $(nm^2q)^{1/(2\beta+2)}$. An additive off-diagonal packing at
	full count is not available: the naive bump construction violates positive
	semi-definiteness for every constant amplitude --- a deterministic Fourier
	obstruction at the alignment frequency $1/h$ --- and the standard repairs each
	break one of the H\"older ball, the cell count, or the likelihood control.
	Whether a positivity-compatible packing of $h^{-2}/q$ off-diagonal cells at
	amplitude $h^\beta$ exists --- which would restore the matching exponent
	$\beta/(\beta+1)$ by the very same Assouad scheme --- is left open. For the
	related problem of estimating the mean function, Cai and Yuan
	\cite{CaiYuan2011} obtain matching bounds by perturbing only the mean while
	keeping the covariance $C_0$ fixed: the KL between hypotheses then reduces to
	$n\sigma_0^{-2}\norm{\mu_b-\mu_{b'}}^2$, and positivity of $C_0$ is immediate.
	That route is not available for covariance estimation, where the object to be
	perturbed is the covariance itself.
\end{remark}

\subsection{A bound uniform in the group}

\begin{theorem}[Universal lower bound]\label{thm:uniform-lb}
	There is $c>0$ as above such that, for every $q\in\Q$,
	\begin{equation}\label{eq:lb-uniform}
		\inf_{\widetilde C_n}\sup_{P\in\mathcal P_q(\beta,R,m)}
		\EE_P\norm{\widetilde C_n-C}_{L^2}^2
		\ge c\Bigl[\bigl(nm^2\bigr)^{-(2\beta+1)/(2\beta+2)}+\frac1n\Bigr].
	\end{equation}
\end{theorem}

\begin{proof}
	By Definition~\ref{def:class} it suffices to prove the bound over $\mathcal
	P_{\mathrm{st}}$, which is contained in $\mathcal P_q$ for every $q$.
	
	Over the stationary sub-class the parameter is the univariate function $\gamma$
	on $\mathbb S^1$, whose non-negative Fourier coefficients are the whole
	positivity constraint --- and the following packing satisfies it by inspection.
	Let $k_0:=\lceil h^{-1}\rceil$ and $M:=k_0$. For $\omega\in\{0,1\}^M$ set
	\[
	\ghat_\omega(k):=\ghat_0(k)+\delta\sum_{\ell=1}^M\omega_\ell\ind\{|k|=k_0+\ell\},
	\qquad C_\omega(s,t):=\gamma_\omega(t-s),
	\]
	with $\gamma_0$ a fixed non-degenerate stationary covariance and $\delta>0$. All
	coefficients are non-negative for every $\omega$, so every $\gamma_\omega$ is a
	covariance: positivity is free, and no smallness of $\delta$ is needed for it.
	
	\emph{H\"older ball.} The worst vertex is $\omega\equiv1$, a Dirichlet block on
	$[k_0,2k_0]$: its $C^\beta$ seminorm is $\asymp\delta Mk_0^\beta\asymp\delta
	h^{-1-\beta}$, so the ball constraint reads $\delta\le R'h^{1+\beta}$.
	
	\emph{Separation.} Distinct frequencies are exactly orthogonal, so for any
	$\omega,\omega'$, $\norm{C_\omega-C_{\omega'}}_{L^2}^2=2\delta^2\rho_H(\omega,\omega')$.
	
	\emph{Information.} For a single flip the covariance perturbation is bounded by
	$\delta$ in supremum norm, whence, on the Gaussian observation model with noise
	floor $\sigma^2$, $\KL(P_\omega\Vert P_{\omega'})\lesssim nm^2\delta^2/\sigma^4$.
	
	\emph{Assouad.} Take $\delta^2\asymp\min\bigl((nm^2)^{-1},R'^2h^{2+2\beta}\bigr)$
	and optimise over $h$: the two constraints meet at
	$h\asymp(nm^2)^{-1/(2\beta+2)}$, where Assouad's lemma gives a lower bound of
	order
	\[
	M\cdot\delta^2\asymp h^{-1}\cdot h^{2+2\beta}=h^{2\beta+1}
	\asymp\bigl(nm^2\bigr)^{-(2\beta+1)/(2\beta+2)}.
	\]
	The floor $n^{-1}$ is obtained as before.
\end{proof}

\subsection{Consequences}

\begin{corollary}[Near-optimality, and the certified part of both thresholds]\label{cor:near-optimal}
	The projected local-polynomial estimator with the bandwidth prescribed in
	Theorems \ref{thm:rate} and \ref{thm:saturated} attains the lower bounds of this
	section up to the factor $(nm^2q)^{1/(2\beta+2)}$ of Remark~\ref{rem:lb-gap}, in
	both regimes and on both sides of the sparse--dense transition. Consequently:
	\begin{enumerate}[label=\textup{(\roman*)},leftmargin=2.2em]
		\item for $q\le\qsh$ the location $\mstar(q)\asymp n^{1/(2\beta)}q^{-1/2}$ of the
		transition is information-theoretically pinned to within that factor;
		\item for every $q$, no estimator whatsoever can reach the parametric regime
		with $m\ll n^{1/(4\beta+2)}$; the exponent $1/(4\beta)$ of
		Corollary~\ref{cor:floor} remains the upper-bound-side threshold, and matching
		it from below is open along with Remark~\ref{rem:lb-gap}.
	\end{enumerate}
\end{corollary}

\begin{proof}
	Compare \eqref{eq:lb-fixed} with \eqref{eq:rate}, and \eqref{eq:lb-uniform} with
	\eqref{eq:saturated-rate}. Part (i) follows by equating the two terms of
	\eqref{eq:lb-fixed}; part (ii) by equating those of \eqref{eq:lb-uniform}, which
	gives $m\asymp n^{1/(4\beta+2)}$.
\end{proof}

\begin{remark}[The two constructions fail at the same place]
	The packing of Theorem~\ref{thm:lb-fixed} requires the $q$ translates of a bump
	to have disjoint supports, that is $qh\lesssim1$; beyond that the summands
	overlap, the $C^\beta$ norm of $\Psi_\ell$ grows with $q$, and the perturbation
	leaves the H\"older ball. This is the same inequality $qh\lesssim c_K$ at which
	the reduction factor of Lemma~\ref{lem:saturation} stops growing. Upper and
	lower bounds therefore degrade simultaneously and for the same reason: past that
	point the group is finer than the resolution, and neither the estimator nor the
	adversary can tell its orbits apart. It is because both fail together that the
	saturated regime needs a construction of a different kind, and the stationary
	sub-class provides it precisely because it collapses the dimension in the same
	way the saturated estimator does.
\end{remark}

\begin{remark}[What is not sharp]
	The upper bound of Theorem~\ref{thm:saturated} carries the kernel constant $c_K$
	whereas \eqref{eq:lb-uniform} does not, the lower bound being estimator-free.
	The two therefore match up to a constant factor $c_K^{2\beta/(2\beta+1)}$, and
	the location $n^{1/(4\beta)}$ of the floor is established in order but not with
	its constant. Obtaining the sharp constant would require an exact asymptotic
	minimax analysis rather than an Assouad argument, and we do not attempt it.
\end{remark}

\section{The phase diagram}\label{sec:diagram}

Everything so far assumed the covariance genuinely invariant. We now drop that
assumption and keep only the invariance of the sampling design, which is
typically a matter of protocol --- regular visit schedules, seasonal monitoring
--- rather than a hypothesis about the process. The symmetry imposed on the
estimator becomes a modelling choice, and its misspecification is what produces
the third regime.

\subsection{Misspecified symmetry: an exact decomposition}

\begin{proposition}[Pythagorean decomposition]\label{prop:pythagoras}
	Let $A_q:=\norm{(I-\Pi_q)C}_{L^2}^2$. Then, almost surely and for every $n$,
	every bandwidth and every $q$,
	\begin{equation}\label{eq:pythagoras}
		\norm{\Chat^{(q)}_{n,h}-C}_{L^2}^2
		=\norm{\Pi_q\Chat_{n,h}-C}_{L^2}^2+A_q.
	\end{equation}
\end{proposition}

\begin{proof}
	Writing $C=\Pi_qC+(I-\Pi_q)C$ gives $\Pi_q\Chat_{n,h}-C=\Pi_q(\Chat_{n,h}-C)-(I-\Pi_q)C$.
	The first term lies in $\mathcal P_{G_q}$ and the second in its orthogonal
	complement, $\Pi_q$ being an orthogonal projection; Pythagoras applies.
\end{proof}

The decomposition is exact, not asymptotic, and it separates the two roles of
$q$ cleanly: the first term is what Sections~\ref{sec:variance} to
\ref{sec:saturation} analysed and decreases with $q$, the second is a
deterministic approximation error and increases with $q$.

On the circle the approximation error is completely explicit. Expanding
$C(s,t)=\sum_{j,k\in\ZZ}\chat_{jk}e^{2\pi i(js+kt)}$, the projection acts as a
Fourier multiplier:
\[
\Pi_qe^{2\pi i(js+kt)}
=\Bigl(\frac1q\sum_{\ell=0}^{q-1}e^{2\pi i(j+k)\ell/q}\Bigr)e^{2\pi i(js+kt)}
=\ind\{q\mid j+k\}\,e^{2\pi i(js+kt)}.
\]

\begin{proposition}[The symmetry spectrum]\label{prop:symmetry-spectrum}
	For every $q$,
	\begin{equation}\label{eq:Aq}
		A_q=\sum_{j+k\notin q\ZZ}|\chat_{jk}|^2,
	\end{equation}
	and consequently
	\begin{enumerate}[label=\textup{(\roman*)},leftmargin=2.2em]
		\item $A_1=0$, and $A_q\le A_{q'}$ whenever $q\mid q'$;
		\item $A_q=0$ if and only if $C$ is $G_q$-invariant;
		\item $A_q\uparrow A_\infty=\sum_{j+k\neq0}|\chat_{jk}|^2$ as $q\to\infty$, and
		$A_\infty=0$ if and only if $C$ is stationary.
	\end{enumerate}
\end{proposition}

\begin{proof}
	Formula \eqref{eq:Aq} is Parseval applied to the multiplier above. Monotonicity
	follows since $q\mid q'$ implies $q'\ZZ\subset q\ZZ$, so the excluded frequency
	set grows with $q$. Part (ii) is $\Pi_qC=C$. For (iii), only the anti-diagonal
	$j+k=0$ belongs to $q\ZZ$ for all $q$, and
	$\sum_j\chat_{j,-j}e^{2\pi ij(s-t)}$ is precisely the stationary part of $C$.
\end{proof}

\begin{remark}[Reading $A_q$]
	The sequence $q\mapsto A_q$ is a fingerprint of the symmetries of $C$: it stays
	at zero as long as $q$ divides the true symmetry order, jumps when it does not,
	and converges to the squared distance from $C$ to stationarity. It is therefore
	the natural object for the selection problem of Section~\ref{sec:selection}, and
	\eqref{eq:Aq} makes it estimable, since the $\chat_{jk}$ are linear functionals
	of $C$.
\end{remark}

\subsection{Three regimes}

\begin{theorem}[Risk under misspecified symmetry]\label{thm:three-terms}
	Assume the design is $G_q$-invariant and let $h$ be chosen as in
	Theorem~\ref{thm:rate} or \ref{thm:saturated} according to the regime. Then
	\begin{equation}\label{eq:three-terms}
		\EE\norm{\Chat^{(q)}_{n,h}-C}_{L^2}^2
		\asymp
		\underbrace{\bigl(nm^2(q\wedge\qsh)\bigr)^{-\beta/(\beta+1)}}_{\text{smoothing}}
		+\underbrace{\frac1n}_{\text{curves}}
		+\underbrace{A_q}_{\text{symmetry}}.
	\end{equation}
\end{theorem}

\begin{proof}
	Take expectations in \eqref{eq:pythagoras}. The first term is the risk of
	$\Pi_q\Chat_{n,h}$ as an estimator of $\Pi_qC$; its bias is
	$\norm{\Pi_q(\EE\Chat_{n,h}-C)}\le\norm{\EE\Chat_{n,h}-C}\asymp h^\beta$ by
	contraction of $\Pi_q$, and its variance is given by
	Proposition~\ref{prop:variance}, which requires only the invariance of the
	design. Optimising over $h$ as in Theorems \ref{thm:rate} and
	\ref{thm:saturated}, and using $r_q(h)\asymp q\wedge\qsh$ at the optimal
	bandwidth, gives the first two terms. The third is $A_q$, which does not depend
	on $h$.
\end{proof}

Expression \eqref{eq:three-terms} partitions the plane of design parameters
$(m,q)$ into three regions according to which term dominates.

\begin{proposition}[The two lines cross at the floor]\label{prop:crossing}
	In the $(\log m,\log q)$ plane the saturation boundary is the line
	$\log q=\frac{2}{2\beta+1}\log m+\log n$ and the transition boundary is
	$\log q=\frac{\beta+1}{\beta}\log n-2\log m$. They intersect at
	$\log m=\frac{\log n}{4\beta}$, that is at $m=\mfloor$.
\end{proposition}

\begin{proof}
	The two expressions come from \eqref{eq:qsharp} and \eqref{eq:threshold}.
	Equating them, $\bigl(\frac{2}{2\beta+1}+2\bigr)\log
	m=\bigl(\frac1\beta-\frac{1}{2\beta+1}\bigr)\log n$, that is
	$\frac{4\beta+4}{2\beta+1}\log m=\frac{\beta+1}{\beta(2\beta+1)}\log n$, whence
	$\log m=\frac{\log n}{4\beta}$.
\end{proof}

\begin{remark}[Geometric reading of the floor]
	Proposition~\ref{prop:crossing} gives a second and more transparent proof of
	Corollary~\ref{cor:floor}. Increasing $q$ moves the transition boundary down and
	to the left, but it also moves the operating point up towards the saturation
	boundary; the best attainable position is the crossing of the two, and that
	crossing is at $n^{1/(4\beta)}$ regardless of anything else. The universal floor
	is where the statistical benefit of symmetry meets its resolution limit.
\end{remark}

\begin{table}[htbp]
	\centering
	\caption{The three regimes and what improves the risk in each.}
	\label{tab:regimes}
	\medskip
	\begin{tabular}{@{}llll@{}}
		\toprule
		Regime & Dominant term & Boundary & What helps\\
		\midrule
		Smoothing-limited & $(nm^2q)^{-\beta/(\beta+1)}$ & $m=\mstar(q)$, $q=\qsh$ & $n$, $m$, $q$\\
		Curve-limited     & $1/n$ & $m=\mstar(q)$, $A_q=1/n$ & $n$ only\\
		Symmetry-limited  & $A_q$ & $A_q=1/n$ & reducing $q$\\
		\bottomrule
	\end{tabular}
\end{table}

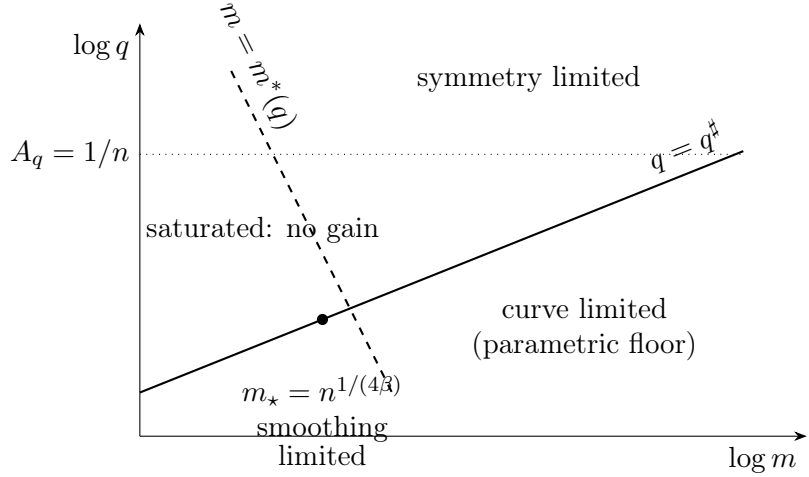
\begin{figure}[htbp]
	\centering
	\begin{tikzpicture}[scale=1.05,>=Stealth]
		\draw[->] (0,0) -- (8.4,0) node[below left] {$\log m$};
		\draw[->] (0,0) -- (0,5.2) node[below left] {$\log q$};
		\draw[thick] (0,0.55) -- (7.6,3.59) node[above,sloped,pos=0.92] {$q=\qsh$};
		\draw[thick,dashed] (1.15,4.6) -- (3.2,0.5) node[above,sloped,pos=0.02] {$m=\mstar(q)$};
		\draw[dotted] (0,3.55) -- (7.6,3.55);
		\node[left] at (0,3.55) {$A_q=1/n$};
		\fill (2.3,1.47) circle (2pt);
		\node[below,align=center] at (2.3,0.9) {$\mfloor=n^{1/(4\beta)}$\\smoothing\\[-2pt]limited};
		\node at (4.9,4.5) {symmetry limited};
		\node[align=center] at (1.55,2.6) {saturated: no gain};
		\node[align=center] at (5.6,1.35) {curve limited\\(parametric floor)};
	\end{tikzpicture}
	\caption{The phase diagram in the $(\log m,\log q)$ plane, drawn for $\beta=2$.
		The saturation line $q=\qsh$ has slope $2/(2\beta+1)$, the transition line
		$m=\mstar(q)$ has slope $-2$, and they cross exactly at $m=\mfloor=n^{1/(4\beta)}$
		(Proposition~\ref{prop:crossing}). The horizontal line depends on $C$ through
		the symmetry spectrum \eqref{eq:Aq} and not on $m$.}
	\label{fig:phase}
\end{figure}

\subsection{The oracle symmetry level}

\begin{corollary}[Oracle choice of $q$]\label{cor:oracle}
	The minimiser of \eqref{eq:three-terms} over $q\in\Q$ is
	\begin{equation}\label{eq:q-oracle}
		q^{\mathrm{or}}
		=\Bigl(\argmin_{q\in\Q}\bigl[(nm^2q)^{-\beta/(\beta+1)}+A_q\bigr]\Bigr)\wedge\qsh.
	\end{equation}
	If moreover the parametric regime is attainable, that is if $m\ge\mstar(q)$ for
	some admissible $q$, then \eqref{eq:q-oracle} reduces to
	\begin{equation}\label{eq:q-oracle-simple}
		q^{\mathrm{or}}=\max\bigl\{q\in\Q:A_q\lesssim1/n\bigr\}\wedge\qsh.
	\end{equation}
\end{corollary}

\begin{proof}
	The first term of \eqref{eq:three-terms} is non-increasing in $q$ and constant
	beyond $\qsh$, while $A_q$ is non-decreasing by
	Proposition~\ref{prop:symmetry-spectrum}(i); the minimum is attained at the
	balance point, truncated at $\qsh$ by Remark~\ref{rem:no-exceed}. When the floor
	$1/n$ is reachable, no further reduction of the first term is worth any increase
	of $A_q$ beyond $1/n$, which gives \eqref{eq:q-oracle-simple}.
\end{proof}

\begin{remark}[The rule, in one line]
	Push the symmetry as far as the parametric floor allows, and no further: $A_q$
	may be spent up to $1/n$ but not beyond, and in any case $q$ should never exceed
	$\qsh$. Two bounds of entirely different natures --- one statistical, the
	sampling floor $1/n$; one geometric, the resolution limit $\qsh$ --- delimit the
	same optimum. The question ``how much symmetry should one impose'' has an answer
	only because the parametric floor exists, which is what makes it a
	functional-data question rather than a nonparametric-regression one.
\end{remark}

\begin{remark}[What the third regime means]
	In the symmetry-limited regime the risk plateaus at $A_q$: collecting more
	observations per curve buys nothing at all, and the dense regime is unreachable
	however large $m$ becomes. By
	Proposition~\ref{prop:symmetry-spectrum}(iii), for large $q$ this plateau
	approaches the squared distance from $C$ to stationarity. The regime is
	therefore not an abstraction: it is what happens when a seasonal or circadian
	symmetry is imposed on a process that does not have it, and its signature --- an
	error curve that stops improving as visits accumulate --- is directly
	observable.
\end{remark}

\begin{table}[htbp]
	\centering
	\caption{Validity of the selection results across the phase diagram.
		\textup{(A4$'$)}: sub-Gaussian innovations (always required); \textup{(G)}:
		concentration on truncated innovations; flatness: eigenvalue spread of
		$\Sigma_P$ on high-frequency modes. ``Automatic'' means the condition follows
		from the rate-optimal choice of $h$.}
	\label{tab:selection-validity}
	\begin{tabular}{@{}llll@{}}
		\toprule
		Region & (G) & Flatness & Oracle guarantee\\
		\midrule
		$m\ll\mstar$ (sparse) & automatic & automatic & constant $1$, full diagram\\
		$m\asymp\mstar$ (transition) & required & automatic except marginal strips
		& constant $1$ under (G)\\
		$m\gg\mstar$ (dense) & required & except marginal strips & absolute constant\\
		\bottomrule
	\end{tabular}
\end{table}

\section{Selecting the symmetry from the data}\label{sec:selection}

The oracle rule \eqref{eq:q-oracle-simple} depends on $C$ through the symmetry
spectrum $A_q$ and is therefore not available. This section shows that the
selection problem it poses is, structurally, the most favourable kind of model
selection there is, gives two procedures, and establishes an oracle inequality
under an explicitly stated concentration condition. That condition is the one
outstanding item of the paper and is discussed at the end. The results of this
section hold on the full phase diagram under \textup{(A1)}--\textup{(A6)} and
\textup{(A4$'$)}; the constant-one guarantee additionally requires condition
\textup{(G)} and the flatness bound away from the marginal strips
($|k|\ge\kappa_0\asymp m^{1/\beta}$) --- both satisfied when $h$ lies in the
natural range prescribed by the rates of Sections~\ref{sec:rates}--%
\ref{sec:saturation}.

\subsection{Why the problem is well posed}

Four properties, all established above, combine to make the selection problem
unusually clean. We record them together because it is their conjunction that
matters.

\begin{proposition}[Structure of the candidate family]\label{prop:family}
	Let $\Q=\{q_0=1<q_1<\dots<q_J\}$ with $q_j\mid q_{j+1}$ and $q_J\le\qsh$.
	Then:
	\begin{enumerate}[label=\textup{(\roman*)}]
		\item \textup{(Nesting.)}
		$P_{G_q}[q_0]\supset P_{G_q}[q_1]\supset\dots\supset P_{G_q}[q_J]$, and
		$\Delta_j:=\Pi_{q_j}-\Pi_{q_{j+1}}$ is the orthogonal projection onto
		$P_{G_q}[q_j]\ominus P_{G_q}[q_{j+1}]$, a span of Fourier modes:
		\[
		P_{G_q}[q_j]\ominus P_{G_q}[q_{j+1}]
		=\operatorname{span}\bigl\{e^{2\pi i(js+kt)}:\ q_j\mid j+k,\ q_{j+1}\nmid
		j+k\bigr\}.
		\]
		\item \textup{(Ordered increments.)}
		$A_{q_j}=\sum_{i<j}\alpha_i$ with
		$\alpha_i:=\norm{\Delta_i C}_{L^2}^2\ge0$; selecting $q$ is exactly choosing
		how many orthogonal frequency blocks to discard.
		\item \textup{(Bias is nearly free of $j$.)}
		The bias of $\Chat^{(q_j)}_{n,h}$ as an estimator of $\Pi_{q_j}C$ is
		$\Pi_{q_j}b_h$ with $\norm{\Pi_{q_j}b_h}\le\norm{b_h}\asymp h^\beta$
		uniformly in $j$; selection at fixed $h$ is therefore a trade-off between
		variance and approximation only.
		\item \textup{(Bounded family.)}
		By Remark~\ref{rem:no-exceed} nothing is lost by truncating at $\qsh$, so
		$J\lesssim\log\qsh$ for a geometric family.
	\end{enumerate}
\end{proposition}

\begin{proof}
	(i) is Proposition~\ref{prop:symmetry-spectrum} together with the fact that a
	difference of nested orthogonal projections is the projection onto the
	orthogonal difference. (ii) follows by telescoping and Parseval. (iii) is the
	contraction property of $\Pi_{q_j}$ applied to $b_h=\EE\Chat_{n,h}-C$. (iv) is
	Remark~\ref{rem:no-exceed}.
\end{proof}

This is ordered variable selection over orthogonal blocks, the setting in which
model selection theory is sharpest~\cite{BirgeMassart2001}. What is not
standard is the risk itself: by \eqref{eq:three-terms} it contains a term $1/n$
that does not scale with the model dimension, and the estimator is a ratio of
within-subject U-statistics rather than a linear projection of independent
observations. Both points return in Section~\ref{subsec:outstanding}.

\subsection{Two procedures}

Write $V(q,h)$ for the integrated variance of $\Chat^{(q)}_{n,h}$ given by
Proposition~\ref{prop:variance},
\[
V(q,h)\asymp\frac{1}{nm^2h^2\,r_q(h)}+\frac1n,
\]
and $\widehat V(q,h)$ for a plug-in version obtained by substituting consistent
estimates of $\sigma^2$, of the design density and of $\Lambda_q$.

\begin{definition}[Penalised selection]\label{def:penalised}
	\[
	\hat q^{\,\pen}:=\argmin_{q\in\Q}\Bigl\{\norm{(I-\Pi_q)\Chat_{n,h}}_{L^2}^2
	+2\,\widehat V(q,h)\Bigr\}.
	\]
\end{definition}

\begin{definition}[Hold-out selection]\label{def:holdout}
	Split the subjects into two halves $S_1,S_2$ of size $n/2$, let $\Chat^{(q)}_1$
	be built from $S_1$ and $\Chat_2$ from $S_2$, and set
	\[
	\hat q^{\,\mathrm{ho}}:=\argmin_{q\in\Q}\;
	\norm{\Chat^{(q)}_1-\Chat_2}_{L^2}^2.
	\]
\end{definition}

The rationale for Definition~\ref{def:penalised} is that
$\norm{(I-\Pi_q)\Chat}^2$ is, up to a term free of $q$, an unbiased estimator
of $A_q$ shifted by the discarded variance; adding twice the retained variance
restores the risk. The rationale for Definition~\ref{def:holdout} is
Lemma~\ref{lem:holdout-target} below.

\begin{lemma}[The hold-out criterion targets the risk]\label{lem:holdout-target}
	Conditionally on $S_2$ being independent of $S_1$,
	\[
	\EE\norm{\Chat^{(q)}_1-\Chat_2}^2
	=\EE\norm{\Chat^{(q)}_1-C}^2
	+\underbrace{\EE\norm{\Chat_2-\EE\Chat_2}^2}_{\textup{free of $q$}}
	+\,r(q,h),\qquad |r(q,h)|\lesssim h^{2\beta}.
	\]
\end{lemma}

\begin{proof}
	Independence gives
	$\EE\norm{\Chat^{(q)}_1-\Chat_2}^2=\EE\norm{\Chat^{(q)}_1-\EE\Chat_2}^2
	+\EE\norm{\Chat_2-\EE\Chat_2}^2$. Since $\EE\Chat_2=C+b_h$,
	$\EE\norm{\Chat^{(q)}_1-\EE\Chat_2}^2
	=\EE\norm{\Chat^{(q)}_1-C}^2-2\EE\inner{\Chat^{(q)}_1-C}{b_h}+\norm{b_h}^2$,
	and both correction terms are $O(h^{2\beta})$ by Cauchy--Schwarz and
	$\norm{b_h}\asymp h^\beta$.
\end{proof}

\begin{remark}[Why the remainder is harmless]\label{rem:harmless}
	The term $r(q,h)$ is of the order of the squared bias, which is already present
	in every risk being compared and is, at the optimal bandwidth, of the same
	order as the risk itself. The hold-out criterion therefore identifies the
	oracle up to a constant factor, not exactly --- which is all
	\eqref{eq:q-oracle-simple} requires, the quantities being compared differing by
	powers of $n$.
\end{remark}

\subsection{An oracle inequality}

\begin{condition}[Uniform concentration of the criterion]\label{cond:concentration}
	There is a sequence $\rho_n(J)$ such that, for the chosen procedure with
	criterion $\mathrm{crit}(q)$,
	\[
	\EE\Bigl[\max_{q\in\Q}\bigl|\mathrm{crit}(q)-\EE\,\mathrm{crit}(q)\bigr|\Bigr]
	\le\rho_n(J).
	\]
\end{condition}

\begin{theorem}[Oracle inequality]\label{thm:oracle}
	Assume Condition~\ref{cond:concentration}. Then for the hold-out procedure, at
	any fixed bandwidth $h$,
	\begin{equation}\label{eq:oracle-ineq}
		\EE\norm{\Chat^{(\hat q^{\,\mathrm{ho}})}_1-C}_{L^2}^2
		\le\min_{q\in\Q}\EE\norm{\Chat^{(q)}_1-C}_{L^2}^2+2\rho_n(J)+ch^{2\beta}.
	\end{equation}
	In particular, if
	$\rho_n(J)=o\bigl(\min_q\EE\norm{\Chat^{(q)}_1-C}^2\bigr)$ and $h$ is chosen
	as in Theorem~\ref{thm:rate} for the oracle level, the selected estimator
	attains the oracle risk of Corollary~\ref{cor:oracle} up to constants, and
	therefore the rate \eqref{eq:three-terms} at $q=q^{\mathrm{or}}$.
\end{theorem}

\begin{proof}
	Write $\mathrm{crit}(q)=\norm{\Chat^{(q)}_1-\Chat_2}^2$ and
	$R(q)=\EE\norm{\Chat^{(q)}_1-C}^2$. By Lemma~\ref{lem:holdout-target},
	$\EE\,\mathrm{crit}(q)=R(q)+\gamma+r(q,h)$ with $\gamma$ free of $q$ and
	$|r|\le c'h^{2\beta}/2$. Let $q^{\mathrm{or}}$ attain the minimum of $R$. By
	definition of $\hat q^{\,\mathrm{ho}}$,
	$\mathrm{crit}(\hat q^{\,\mathrm{ho}})\le\mathrm{crit}(q^{\mathrm{or}})$.
	Adding and subtracting the expectations,
	\[
	R(\hat q^{\,\mathrm{ho}})+\gamma+r(\hat q^{\,\mathrm{ho}},h)
	\le R(q^{\mathrm{or}})+\gamma+r(q^{\mathrm{or}},h)
	+2\max_q\bigl|\mathrm{crit}(q)-\EE\,\mathrm{crit}(q)\bigr|.
	\]
	Cancelling $\gamma$, bounding the two remainders by $c'h^{2\beta}/2$ each and
	taking expectations, Condition~\ref{cond:concentration} gives
	\eqref{eq:oracle-ineq} with $c=c'$.
\end{proof}

\begin{remark}[What the theorem does and does not deliver]
	The argument above is elementary and complete: the whole difficulty of the
	selection problem has been pushed into Condition~\ref{cond:concentration},
	which is where it belongs. What remains is to exhibit an admissible
	$\rho_n(J)$, and that is a genuine piece of work, discussed next. We state the
	theorem in this conditional form rather than assert a rate we have not
	established.
\end{remark}

\subsection{The outstanding condition}\label{subsec:outstanding}

Condition~\ref{cond:concentration} is not a technicality, and two features of
the present problem put it outside the reach of the standard toolbox.

\subsubsection{The criterion is a degenerate U-statistic of order four}
The estimator $\Chat_{n,h}$ is a ratio whose numerator is a sum, over subjects,
of within-subject sums over pairs $j\ne k$; the criterion is a quadratic
functional of it, hence a U-statistic of order four in the observations,
degenerate after centring, with a kernel that depends on $h$ and concentrates
as $h\to0$. Exponential bounds for such objects exist
\cite{GineLatalaZinn2000,HoudreReynaud2003}, but they involve four distinct
norms of the kernel, each of which must be evaluated at the $h$-dependent
scale. Moment assumption \textup{(A4)}, which controls only the marginals of
$X$, is very probably insufficient; sub-Gaussianity of $X$ and of
$\varepsilon$, or at least Bernstein-type moment growth, is the natural
strengthening.

\subsubsection{The parametric floor is not proportional to the model dimension}
Classical penalised selection assumes a risk of the form
$\mathrm{bias}^2+\sigma^2D_j/N$, so that the penalty can be taken proportional
to the dimension. Here the relevant complexity is the effective dimension
(stable rank) of the block fluctuation covariance,
$\Eff(\Sigma_P)=(\tr\Sigma_P)^2/\norm{\Sigma_P}_{\HS}^2$ with
$\Sigma_P=\EE[P\xi\otimes P\xi]$. The Fourier-mode count $h^{-2}/r_q(h)$ is
only an upper bound for it: the per-subject fluctuation $\xi_i$ is a
superposition of $\asymp N_i^2$ kernel bumps at the pairs $(T_{ij},T_{ik})$,
which share their coordinates, so the number of effectively distinct directions
is governed by one dimension of observation times rather than two. Measured on
the design of Section~\ref{sec:numerical}
(Section~\ref{subsec:deff-saturation}), $\Eff\asymp h^{-1.6}q^{-0.8}$ below
saturation and $\Eff=O(1)$ at $q=\qsh$. The risk carries in addition the term
$1/n$ of Proposition~\ref{prop:variance}, which is insensitive to $j$ up to a
bounded constant by Corollary~\ref{cor:floor}. The minimal-penalty condition of
Birg\'e--Massart must therefore be re-established rather than invoked. The
hold-out procedure of Definition~\ref{def:holdout} was included precisely
because it sidesteps this difficulty: it requires no penalty at all, only
Condition~\ref{cond:concentration}.

\subsubsection{What is plausible}
For a geometric family the block orthogonality of
Proposition~\ref{prop:family}(i) should yield a union bound in $\log J$ rather
than $J$, and $J\lesssim\log\qsh\lesssim\log n$ by
Proposition~\ref{prop:family}(iv), so a remainder
$\rho_n(J)\lesssim V(q_J,h)\sqrt{\log\log n}$ is the natural target. That would
make \eqref{eq:oracle-ineq} adaptive up to a $\sqrt{\log\log n}$ factor, which
is negligible against the polynomial gaps that separate the three regimes of
Section~\ref{sec:diagram}. This is carried out below: the outcome is
Theorem~\ref{thm:ustat-holds}, whose remainder is $\sqrt{\log\log n}$ on flat
blocks and $\log\log n$ in general.

\begin{remark}[A fallback that requires nothing]\label{rem:fallback}
	Independently of Condition~\ref{cond:concentration}, the deterministic
	contraction of Proposition~\ref{prop:pythagoras} guarantees that any selected
	level $\hat q$ satisfies
	$\norm{\Chat^{(\hat q)}-C}^2\le\norm{\Chat-C}^2+A_{\hat q}$ almost surely.
	Selecting conservatively --- taking the largest $q$ for which a pre-specified
	symmetry is scientifically credible, rather than estimating $A_q$ ---
	therefore never does worse than the unprojected estimator by more than the
	modelling error one has knowingly accepted. This is not adaptivity, but it is a
	usable guarantee, and it holds with no assumption whatsoever beyond measure
	preservation.
\end{remark}

\subsection{Establishing Condition~\ref{cond:concentration}}

\subsubsection*{Two reductions}
Condition~\ref{cond:concentration} concerns a maximum over the candidate
family of a quadratic functional of two dependent estimators. Two elementary
reductions turn it into a much smaller problem: the first removes the family,
the second removes the within-subject dependence.

\begin{lemma}[Block reduction]\label{lem:block-reduction}
	For $j<j'$, let $P:=\Pi_{q_j}-\Pi_{q_{j'}}$, the orthogonal projection onto
	$P_{G_{q_j}}\ominus P_{G_{q_{j'}}}$. Then the hold-out criterion of
	Definition~\ref{def:holdout} satisfies
	\begin{equation}\label{eq:crit-diff}
		\mathrm{crit}(q_j)-\mathrm{crit}(q_{j'})
		=\norm{P\Chat_1}^2-2\inner{P\Chat_1}{P\Chat_2}.
	\end{equation}
	In particular the comparison of any two levels depends on the data only through
	the single block projection $P$, and the whole selection is governed by the $J$
	statistics obtained from $P=\Pi_{q_0}-\Pi_{q_{j'}}$, $j'=1,\dots,J$.
\end{lemma}

\begin{proof}
	Write $\Pi_{q_j}=\Pi_{q_{j'}}+P$ with $P\Pi_{q_{j'}}=0$. Then
	\[
	\norm{\Pi_{q_j}\Chat_1-\Chat_2}^2
	=\norm{\Pi_{q_{j'}}\Chat_1-\Chat_2}^2
	+2\inner{P\Chat_1}{\Pi_{q_{j'}}\Chat_1-\Chat_2}
	+\norm{P\Chat_1}^2,
	\]
	and $\inner{P\Chat_1}{\Pi_{q_{j'}}\Chat_1}=0$ by orthogonality of the ranges.
	Self-adjointness and idempotence of $P$ give
	$\inner{P\Chat_1}{\Chat_2}=\inner{P\Chat_1}{P\Chat_2}$.
\end{proof}

\begin{lemma}[Subject decoupling]\label{lem:subject-decoupling}
	Write $\mu_h:=\EE\Chat_{n,h}$ and let
	$\Chat_a-\mu_h=\bar\xi_a+R_a$, $a=1,2$, where
	$\bar\xi_a=(2/n)\sum_{i\in S_a}\xi_i$ is the average of the i.i.d.\ centred
	subject contributions
	\[
	\xi_i:=\frac{1}{\EE D_n}\sum_{j\ne k}
	\bigl[W_{ijk}Z_{ijk}-\EE W_{ijk}Z_{ijk}\bigr]\in L^2(E\times E),
	\]
	and $R_a$ collects the linearisation remainder. Then the centred version of
	\eqref{eq:crit-diff} decomposes as
	\begin{equation}\label{eq:decomp}
		\underbrace{-2\inner{P\mu_h}{P\bar\xi_2}}_{\textup{(I)}}
		+\underbrace{\frac{4}{n^2}\sum_{i\in S_1}
			\bigl[\norm{P\xi_i}^2-\EE\norm{P\xi_i}^2\bigr]}_{\textup{(IId)}}
		+\underbrace{\frac{4}{n^2}\sum_{i\ne i'\in S_1}
			\inner{P\xi_i}{P\xi_{i'}}}_{\textup{(IIo)}}
		\underbrace{-\,2\inner{P\bar\xi_1}{P\bar\xi_2}}_{\textup{(III)}}
	\end{equation}
	up to the remainder terms involving $R_1,R_2$.
\end{lemma}

\begin{proof}
	Substitute $\Chat_a=\mu_h+\bar\xi_a+R_a$ into \eqref{eq:crit-diff}. The terms
	$\norm{P\mu_h}^2$ cancel between the two brackets, the terms
	$2\inner{P\mu_h}{P\bar\xi_1}$ cancel likewise, and what survives is
	$-2\inner{P\mu_h}{P\bar\xi_2}+\norm{P\bar\xi_1}^2
	-2\inner{P\bar\xi_1}{P\bar\xi_2}$. Expanding $\norm{P\bar\xi_1}^2$ into its
	diagonal and off-diagonal parts gives \eqref{eq:decomp}.
\end{proof}

The decomposition is the point of the section. The estimator was a ratio of
within-subject U-statistics, and the criterion a quadratic functional of it, so
that Condition~\ref{cond:concentration} appeared to concern a degenerate
U-statistic of order four with an $h$-dependent kernel.
Lemma~\ref{lem:subject-decoupling} shows that the within-subject structure is
entirely absorbed into the single Hilbert-valued variable $\xi_i$, and that
what remains is of order two at most, in $n/2$ independent summands.

\subsubsection*{Three terms out of four are elementary}

\begin{proposition}[Terms \textup{(I)}, \textup{(IId)} and
	\textup{(III)}]\label{prop:easy-terms}
	Let $D:=\Eff(\Sigma_P)=(\tr\Sigma_P)^2/\norm{\Sigma_P}_{\HS}^2$ denote the
	effective dimension of the block fluctuation covariance
	$\Sigma_P:=\EE[P\xi\otimes P\xi]$, and $V_P:=\EE\norm{P\bar\xi_1}^2$. Under
	\textup{(A1)}--\textup{(A6)} strengthened by sub-Gaussianity of $X$ and
	$\varepsilon$, each of the three terms \textup{(I)}, \textup{(IId)} and
	\textup{(III)} of \eqref{eq:decomp} satisfies
	\[
	\EE\Bigl[\max_{j'=1,\dots,J}\bigl|\,\cdot\,\bigr|\Bigr]
	\lesssim V_P\frac{\sqrt{\log J}}{\sqrt D}
	+\norm{P\mu_h}\sqrt{V_P\log J}.
	\]
\end{proposition}

\begin{proof}
	(I). The quantity $\inner{P\mu_h}{P\bar\xi_2}$ is
	$(2/n)\sum_{i\in S_2}\inner{P\mu_h}{P\xi_i}$, a sum of i.i.d.\ centred real
	variables with variance $\norm{P\mu_h}^2V_P$ at most, by Cauchy--Schwarz, and
	sub-exponential tails under the strengthened moment assumption. Bernstein's
	inequality and a union bound over $J$ values give the second summand of the
	bound.
	
	(III). Conditionally on $S_2$, the quantity $\inner{P\bar\xi_1}{P\bar\xi_2}$ is
	again a sum of i.i.d.\ centred real variables over $S_1$, with conditional
	variance $(2/n)\norm{P\bar\xi_2}^2V_P/V_P$; Bernstein applies conditionally,
	and integrating over $S_2$, whose $\norm{P\bar\xi_2}^2$ concentrates around
	$V_P$, gives a bound of order $V_P\sqrt{\log J}/\sqrt D$ --- the factor
	$D^{-1/2}$ arising because
	$\EE\inner{P\xi}{P\xi'}^2=\norm{\EE[P\xi\otimes P\xi]}_{\HS}^2
	=(\EE\norm{P\xi}^2)^2/D$ exactly, by definition of the effective dimension ---
	no isotropy is required.
	
	(IId). This is again a sum of i.i.d.\ centred real variables, namely
	$\norm{P\xi_i}^2-\EE\norm{P\xi_i}^2$, rescaled by $4/n^2$. Its standard
	deviation is $n^{-3/2}$ times that of a single term and is therefore smaller by
	a factor $n^{-1/2}$ than (III); Bernstein and a union bound suffice.
\end{proof}

\subsubsection*{The one remaining term}

\begin{condition}[Degenerate U-statistic bound]\label{cond:ustat}
	With $P$, $D$ and $V_P$ as above,
	\[
	\EE\Bigl[\max_{j'=1,\dots,J}\Bigl|\frac{4}{n^2}\sum_{i\ne i'\in S_1}
	\inner{P\xi_i}{P\xi_{i'}}\Bigr|\Bigr]
	\lesssim V_P\frac{\sqrt{\log J}}{\sqrt D}.
	\]
\end{condition}

This is a degenerate, real-valued U-statistic of order two built on $n/2$
i.i.d.\ Hilbert-valued variables with kernel
$g(\xi,\xi')=\inner{P\xi}{P\xi'}$. Exponential inequalities for exactly this
object are available \cite{GineLatalaZinn2000,HoudreReynaud2003}; they involve
four norms of $g$, namely $\norm{\EE[P\xi\otimes P\xi]}_{\HS}$, the operator
norm of the same, a sup-norm of $g$ and a mixed norm. The first is
$\asymp\EE\norm{P\xi}^2/\sqrt D$ as used above; the obstacle is that the last
two are governed by $\norm{\xi_i}_{L^2}\lesssim h^{-1}\norm{Z}_\infty$, which
diverges as $h\to0$, so the bound must be shown to be dominated by the first
two rather than by the sup-norm terms.

We now carry the route out. Throughout, $N=n/2$ is the size of $S_1$, the
$\xi_i$, $i\in S_1$, are i.i.d.\ centred $L^2$-valued fluctuations,
$\{P_a\}_{a\le\bar J}$ is the family of block projections produced by
Lemma~\ref{lem:block-reduction} ($\bar J\le J^2$, so
$\bar L:=\log(e\bar J)\asymp\log(eJ)$), and for an orthogonal projection $P$
we write
\[
\Sigma_P:=\EE\bigl[P\xi\otimes P\xi\bigr],\qquad
T_P:=\tr\Sigma_P,\qquad
D_P:=\frac{T_P^2}{\norm{\Sigma_P}_{\HS}^2},\qquad
\kappa_P:=\frac{\norm{\Sigma_P}_{\op}}{\norm{\Sigma_P}_{\HS}}\le1.
\]
Two assumptions are added to the standing set.

\medskip
\noindent\textbf{Assumption (A4$'$).}\enspace
\emph{The process admits a Karhunen--Lo\`eve representation
	$X=\sum_{k\ge1}\sqrt{\lambda_k}\,\eta_k\phi_k$ whose coordinates $\eta_k$ are
	independent and uniformly sub-Gaussian, and the errors $\varepsilon_{ij}$ are
	i.i.d.\ sub-Gaussian, independent of $X$ and of the design.} A Gaussian
process with square-integrable paths satisfies this with $\eta_k$ i.i.d.\
standard normal. The independence of the innovations is not a stylistic choice:
the observed vector $Y_i=(X_i(T_{i1})+\varepsilon_{i1},\dots)$ has dependent
coordinates --- their dependence is the estimand --- and the Hanson--Wright
inequality on which Lemma~\ref{lem:tail} rests is a statement about quadratic
forms in independent sub-Gaussian variables \cite{RudelsonVershynin2013};
sub-Gaussianity of the marginals of $X$, or even of $\sup_t|X(t)|$, is not
known to suffice, all available proofs requiring either independent innovations
or a convex concentration property that marginals do not provide. Under
\textup{(A4$'$)}, conditionally on the design, $Y_i=M_iz$ for a linear map
$M_i$ with $M_i^\top M_i$ of trace and operator norm at most $CN_i$, and $z$ a
vector of independent uniformly sub-Gaussian coordinates --- which is exactly
the situation the inequality covers.

\medskip
\noindent\textbf{Assumption (G).}\enspace
\emph{With $\tau_n=\tau_{n,20}$ as in Lemma~\ref{lem:tail},
	\[
	\tau_n^2\,\bar L^2\le c_1\,N\norm{\Sigma_{P_a}}_{\HS}
	\]
	for every block $P_a$ of the family.}

\begin{remark}[What \textup{(G)} costs]\label{rem:G-cost}
	Since $T_{P_a}\asymp c_a\nu_2h^{-2}$ with $c_a$ the block's share of the local
	variance, and $\norm{\Sigma_{P_a}}_{\HS}=T_{P_a}/\sqrt{D_{P_a}}$ with
	$D_{P_a}\le Ch^{-2}$, condition \textup{(G)} reduces to
	$n\ge C\,h^{-1}(1+\log^4 n/m^2)\log^8 n/c_a$: a bandwidth in the family may
	not be smaller than $\log^{12}n/n$. At the working bandwidths \eqref{eq:hq}
	and \eqref{eq:hsharp}, over the whole diagram $m\le\mstar(1)$, $q\le\qsh$, one
	has $h^{-1}\le(nm^2q)^{1/(2\beta+2)}\le n^{(\beta+1)/(\beta(2\beta+1))}$,
	whose exponent is $<1$ exactly when $\beta>1/\sqrt2$; \textup{(G)} is
	therefore automatic for $\beta\ge3/4$, and for
	$\beta\in(1/2,1/\sqrt2]$ it excises only the extreme corner $m\asymp\mstar(1)$,
	$q\asymp\qsh$.
\end{remark}

\begin{lemma}[Polynomially negligible tail of the fluctuation
	norm]\label{lem:tail}
	Assume \textup{(A1)}--\textup{(A6)} and \textup{(A4$'$)}. For every $A>0$
	there is a constant $C_A$ such that, with
	\[
	\tau_{n,A}:=C_A\,h^{-1}\bigl(m+\log^2 n\bigr)\log^3 n,
	\]
	one has $\PP\bigl(\max_{i\le n}\norm{\xi_i}_{L^2}>\tau_{n,A}\bigr)\le n^{-A}$.
	Moreover, for every fixed $r\ge1$,
	\begin{equation}\label{eq:trunc}
		\EE\bigl[\norm{\xi}^r\ind\{\norm{\xi}>\tau_{n,A}\}\bigr]
		\le\bigl(\EE\norm{\xi}^{2r}\bigr)^{1/2}n^{-A/2},
		\qquad
		\EE\norm{\xi}^{2r}\le C_r\,h^{-2r}(m+1)^{6r}.
	\end{equation}
	Throughout the paper we take $A=20$ and write $\tau_n:=\tau_{n,20}$.
\end{lemma}

\begin{proof}
	Condition on the design of subject $i$ and write $Y_i=(Y_{i1},\dots,Y_{iN_i})$,
	$\Gamma_i=\Cov(Y_i\mid\text{design})=\bigl(C(T_{ij},T_{ik})\bigr)_{jk}
	+\sigma^2I$, so that $\tr\Gamma_i\le CN_i$ and
	$\norm{\Gamma_i}_{\op}\le N_i\norm{C}_\infty+\sigma^2$. Two elementary bounds
	do the work; no vector-valued chaos inequality is needed.
	
	(i) \emph{A deterministic Gram bound.} With
	$B_{jk}:=K_h(T_{ij}-\cdot)K_h(T_{ik}-\cdot)$ the bump surfaces,
	$\norm{B_{jk}}_{L^2}=h^{-1}R(K)$, and $\inner{B_{jk}}{B_{lp}}=0$ unless
	$|T_{ij}-T_{il}|\le2h$ and $|T_{ik}-T_{ip}|\le2h$. Let $E_1$ be the event that
	every window of length $4h$ contains at most $C(1+N_ih+\log n)$ observation
	times of each subject; since the design density is bounded, binomial tails give
	$\PP(E_1^c)\le n^{-10}$. On $E_1$ each row of the Gram matrix of the $B_{jk}$
	has at most $C(1+N_ih+\log n)^2$ nonzero entries, each at most $h^{-2}R(K)^2$,
	so $\norm{\mathrm{Gram}}_{\op}\le Ch^{-2}(1+N_ih+\log n)^2$ and, writing
	$\xi_i=\sum_{j\ne k}c_{jk}B_{jk}$ with $c_{jk}=Y_{ij}Y_{ik}-C(T_{ij},T_{ik})$,
	\[
	\norm{\xi_i}\le\norm{\mathrm{Gram}}_{\op}^{1/2}\norm{c}_{\ell^2}
	\le C\,\frac{1+N_ih+\log n}{h}\,
	\bigl(\norm{Y_i}^2+\norm{C}_\infty N_i\bigr),
	\]
	since $\norm{c}_{\ell^2}\le\sqrt2\norm{Y_i}^2+\sqrt2\norm{C}_\infty N_i$.
	
	(ii) \emph{Scalar Hanson--Wright on $\norm{Y_i}^2$.} By \textup{(A4$'$)},
	conditionally on the design $Y_i=M_iz$ with $z$ of independent, uniformly
	sub-Gaussian coordinates and $\tr(M_i^\top M_i)=\tr\Gamma_i\le CN_i$,
	$\norm{M_i^\top M_i}_{\op}=\norm{\Gamma_i}_{\op}\le CN_i$, hence also
	$\norm{M_i^\top M_i}_{\HS}\le CN_i$; a finite-dimensional truncation of $z$
	changes none of these constants. The Hanson--Wright inequality
	\cite[Thm~1.1]{RudelsonVershynin2013} applied to
	$\norm{Y_i}^2=z^\top M_i^\top M_iz$ gives
	\[
	\PP\bigl(\norm{Y_i}^2>CN_i+t\bigm|\text{design}\bigr)
	\le2\exp\Bigl(-c\min\Bigl\{\frac{t^2}{N_i^2},\frac{t}{N_i}\Bigr\}\Bigr),
	\]
	and $t=CN_i\log^2 n$ makes the right side $\le2e^{-c\log^2 n}$. All three
	events --- Poisson ($\{N_i\le\bar m:=m+C_A(\sqrt m+1)\log n\ \forall i\}$),
	crowding ($E_1$, with the constant $C_A$) and the Hanson--Wright deviation at
	$t=C_AN_i\log^2 n$, whose conditional exponent $e^{-cC_A\log^2 n}$ is
	super-polynomial --- can be driven below $n^{-A}/3$ by the choice of $C_A$
	alone: the exponent $A$ is free, only the constant pays. Intersecting (i) and
	(ii) on these events and taking a union over $i\le n$ proves the probability
	bound: the Poisson inflation obeys $\bar m\le C_A(m+\log^2 n)$ by AM--GM
	($\sqrt m\log n\le\tfrac12(m+\log^2 n)$), the crowding factor contributes one
	power of $\log n$, Hanson--Wright two, and $(1+\bar mh)=O(1)$ at every working
	bandwidth of the paper since $mh\le1$ there (equivalently $m^{2\beta}\le nq$,
	true throughout the diagram $m\le\mstar(1)$). The first bound of
	\eqref{eq:trunc} is Cauchy--Schwarz against the probability bound; the second
	follows from the deterministic inequality
	$\norm{\xi_i}\le Ch^{-1}N_i(\norm{Y_i}^2+CN_i)$ --- the Gram bound with the
	trivial row count $N_i^2$ --- together with Poisson and sub-Gaussian moment
	bounds. For Gaussian $X$ one may instead conclude by Gaussian-chaos
	concentration \cite[Ch.~6]{BoucheronLugosiMassart2013}; the route above has the
	advantage of covering the sub-Gaussian case with a scalar inequality only.
\end{proof}

\begin{lemma}[The four norms of the projected kernel]\label{lem:four-norms}
	Let $\zeta$ be a centred $L^2$-valued variable with $\norm{\zeta}\le2\tau$
	a.s.\ and covariance $\Sigma_\zeta\preceq\Sigma_P$, and let
	$g(x,y):=\inner{Px}{Py}$. Then $g$ is canonical ($\EE g(x,\zeta)=0$ for every
	$x$) and
	\begin{align*}
		\bigl(\EE g^2(\zeta,\zeta')\bigr)^{1/2}&=\norm{\Sigma_\zeta}_{\HS}
		\le\norm{\Sigma_P}_{\HS}, &
		\norm{T_g}_{L^2\to L^2}&=\norm{\Sigma_\zeta}_{\op}
		\le\norm{\Sigma_P}_{\op},\\
		\sup_{\norm{x}\le2\tau}\bigl(\EE g^2(x,\zeta)\bigr)^{1/2}
		&\le2\tau\norm{\Sigma_P}_{\op}^{1/2}, &
		\norm{g}_\infty&\le4\tau^2.
	\end{align*}
\end{lemma}

\begin{proof}
	Canonicity is the centring of $\zeta$. For the first norm,
	$\EE g^2=\EE_{\zeta'}\inner{\Sigma_\zeta P\zeta'}{P\zeta'}=\tr\Sigma_\zeta^2$;
	monotonicity of $\tr(\cdot)^2$ on the positive order follows from
	$\tr A^2\le\tr AB\le\tr B^2$ for $0\preceq A\preceq B$. For the operator norm,
	$\inner{f}{T_gf}=\EE\bigl[f(\zeta)f(\zeta')\inner{P\zeta}{P\zeta'}\bigr]
	=\norm{\EE[f(\zeta)P\zeta]}^2$; writing $v=\EE[fP\zeta]$, for any unit $u$,
	$\inner{v}{u}=\EE\bigl[f\inner{P\zeta}{u}\bigr]
	\le\inner{\Sigma_\zeta u}{u}^{1/2}\norm{f}_{L^2}$, with equality for
	$f\propto\inner{P\zeta}{u_1}$ at the top eigenvector: the supremum is exactly
	$\norm{\Sigma_\zeta}_{\op}$. The mixed norm is
	$\sup_x\inner{\Sigma_\zeta Px}{Px}^{1/2}
	\le\norm{\Sigma_\zeta}_{\op}^{1/2}\cdot2\tau$, and the sup norm is
	Cauchy--Schwarz.
\end{proof}

\begin{lemma}[Integrating a four-regime tail over a family]\label{lem:integrate}
	Suppose real variables $V_1,\dots,V_{\bar J}$ satisfy, for all $t>0$,
	$\PP(|V_a|\ge t)\le C\exp\bigl(-\frac1C\min\{(t/A)^2,t/B,(t/C)^{2/3},
	(t/D)^{1/2}\}\bigr)$. Then
	\[
	\EE\max_{a\le\bar J}|V_a|
	\le C'\bigl(A\sqrt{\bar L}+B\bar L+C\bar L^{3/2}+D\bar L^2\bigr),
	\qquad \bar L=\log(e\bar J).
	\]
\end{lemma}

\begin{proof}
	Let $t^*$ be the bracket on the right with $C'=1$. For $s\ge1$, each of the
	four regime functions evaluated at $st^*$ is at least $s^{1/2}\bar L$ (the
	worst power being $1/2$), so
	$\PP(\max_a|V_a|\ge st^*)\le\bar JCe^{-s^{1/2}\bar L/C}
	\le Ce^{-(s^{1/2}-C)\bar L/C}$, and
	$\EE\max\le t^*+t^*\int_1^\infty Ce^{-(s^{1/2}-C)\bar L/C}\,ds\le C''t^*$.
\end{proof}

\begin{theorem}[Condition~\ref{cond:ustat} holds]\label{thm:ustat-holds}
	Assume \textup{(A1)}--\textup{(A6)}, \textup{(A4$'$)} and \textup{(G)}. Then
	\begin{equation}\label{eq:ustat-bound}
		\EE\max_{a\le\bar J}\Bigl|\frac{4}{n^2}\sum_{i\ne i'\in S_1}
		\inner{P_a\xi_i}{P_a\xi_{i'}}\Bigr|
		\le C\max_{a\le\bar J}\frac{\norm{\Sigma_{P_a}}_{\HS}}{N}
		\bigl(\sqrt{\bar L}+\kappa\,\bar L\bigr)+CV_1n^{-4},
	\end{equation}
	where $\kappa:=\max_a\kappa_{P_a}\le1$ and $V_1=\EE\norm{\xi}^2$. In
	particular, on any sub-family whose blocks satisfy the flatness bound
	$\kappa_{P_a}\le(2\bar L)^{-1/2}$ --- discharged, off the marginal strips and
	under the window \eqref{eq:window}, by Lemma~\textup{(U1$''$)} and its
	Corollary below --- the leading term is the Gaussian one and
	Condition~\ref{cond:ustat} holds as stated, since
	$\norm{\Sigma_P}_{\HS}/N=V_P/\sqrt{D_P}$ with $V_P=T_P/N$; in general it holds
	with $\sqrt{\log J}$ replaced by $\log J$.
\end{theorem}

\begin{proof}
	Fix the truncation $\xi_\le:=\xi\ind\{\norm{\xi}\le\tau_n\}$,
	$\mu:=\EE\xi_\le$, $\zeta:=\xi_\le-\mu$, and split each summand as
	$\inner{P\xi_i}{P\xi_{i'}}=\inner{P\zeta_i}{P\zeta_{i'}}
	+\inner{P\zeta_i}{P\mu}+\inner{P\mu}{P\zeta_{i'}}+\norm{P\mu}^2+\rho_{ii'}$,
	where $\rho_{ii'}$ collects every term containing a factor
	$\ind\{\norm{\xi}>\tau_n\}$.
	
	\emph{Remainders.} Since $|\inner{Pa}{Pb}|\le\norm{a}\norm{b}$ uniformly in
	$P$, Cauchy--Schwarz and Lemma~\ref{lem:tail} give
	$\EE\max_a\frac4{n^2}\bigl|\sum\rho_{ii'}\bigr|
	\le8V_1^{1/2}\bigl(\EE\norm{\xi}^2\ind\{\norm{\xi}>\tau_n\}\bigr)^{1/2}
	\le CV_1n^{-4}$, the last step by \eqref{eq:trunc} with $r=2$ and $A=20$:
	$\EE\norm{\xi}^2\ind\{\norm{\xi}>\tau_n\}\le Ch^{-2}(m+1)^6n^{-10}
	\le V_1n^{-8}$, since $(m+1)^6/m^2\le Cn^2$. The case $r=1$ of
	\eqref{eq:trunc} yields
	$\norm{\mu}\le\EE\norm{\xi}\ind\{\norm{\xi}>\tau_n\}\le V_1^{1/2}n^{-8}$, so
	the constant term $\norm{P\mu}^2$ and the two linear terms --- centred i.i.d.\
	sums with increments bounded by $2\tau_n\norm{\mu}$ and variance
	$\le\norm{\Sigma_P}_{\op}\norm{\mu}^2/N$, maximised over $\bar J$ blocks by
	Bernstein and a union bound --- are all $O(V_1n^{-4})$.
	
	\emph{Canonical part.} The variables $\zeta_i$ are i.i.d.\ and centred with
	$\norm{\zeta}\le2\tau_n$, and for every orthogonal projection $P$,
	\[
	\Sigma_{\zeta,P}:=\EE[P\zeta\otimes P\zeta]
	=\EE[P\xi_\le\otimes P\xi_\le]-P\mu\otimes P\mu
	\preceq\EE[P\xi\otimes P\xi]=\Sigma_P,
	\]
	both subtractions contracting the positive order. The kernel
	$g_a(x,y)=\inner{P_ax}{P_ay}$ is canonical with respect to the law of $\zeta$,
	and Lemma~\ref{lem:four-norms} gives the four parameters
	\[
	A_a=N\norm{\Sigma_{P_a}}_{\HS},\qquad
	B_a=N\norm{\Sigma_{P_a}}_{\op},\qquad
	C_a=2\sqrt N\,\tau_n\norm{\Sigma_{P_a}}_{\op}^{1/2},\qquad
	D_a=4\tau_n^2.
	\]
	Hence the exponential inequality for canonical bounded U-statistics of order
	two in its four-parameter form --- \cite[Thm~3.4]{HoudreReynaud2003}, whose
	martingale proof applies to the undecoupled statistic directly and carries
	explicit constants, any fixed $\varepsilon$ in their $\kappa(\varepsilon)$
	being absorbed into the constant $C$ below; the form originates in
	\cite[Thm~3.3 and Cor.~3.4]{GineLatalaZinn2000}, where the transfer from the
	decoupled case uses \cite{DelaPenaMontgomerySmith1995} --- followed by
	Lemma~\ref{lem:integrate} with each parameter maximised over the family, yields
	\[
	\EE\max_{a\le\bar J}\Bigl|\sum_{i<i'}g_a(\zeta_i,\zeta_{i'})\Bigr|
	\le C\max_{a\le\bar J}\bigl[A_a\sqrt{\bar L}+B_a\bar L+C_a\bar L^{3/2}
	+D_a\bar L^2\bigr].
	\]
	Since $n=2N$ and $\sum_{i\ne i'}g_{ii'}=2\sum_{i<i'}g_{ii'}$, the required
	normalisation is $2/N^2$. Assumption \textup{(G)} gives
	$\tau_n^2\bar L^2\le c_1N\norm{\Sigma_{P_a}}_{\HS}$, and therefore
	\[
	\frac{\sqrt N\,\tau_n\norm{\Sigma_{P_a}}_{\op}^{1/2}\bar L^{3/2}}{N^2}
	\le\sqrt{c_1}\,
	\frac{\bigl(\norm{\Sigma_{P_a}}_{\HS}\norm{\Sigma_{P_a}}_{\op}\bigr)^{1/2}}{N}
	\sqrt{\bar L}
	\le\sqrt{c_1}\,\frac{\norm{\Sigma_{P_a}}_{\HS}}{N}\sqrt{\bar L},
	\]
	where we used $\norm{\Sigma_{P_a}}_{\op}\le\norm{\Sigma_{P_a}}_{\HS}$. The
	fourth-regime contribution is even smaller:
	$\tau_n^2\bar L^2/N^2\le c_1\norm{\Sigma_{P_a}}_{\HS}/N$. Consequently, with
	$\kappa=\max_a\kappa_{P_a}$,
	\[
	\EE\max_{a\le\bar J}\bigl|U^{\mathrm{can}}_a\bigr|
	\le C\max_{a\le\bar J}\frac{\norm{\Sigma_{P_a}}_{\HS}}{N}
	\bigl(\sqrt{\bar L}+\kappa\,\bar L\bigr).
	\]
	What remains is \eqref{eq:ustat-bound}; the two displayed consequences are the
	cases $\kappa\le(2\bar L)^{-1/2}$ and $\kappa\le1$.
\end{proof}

\begin{remark}[Which form of the inequality is needed where]\label{rem:which-form}
	The chapter of \cite{GineLatalaZinn2000} also records, for generalized kernels
	$h_{i,j}$, a simpler three-parameter form of the inequality, with exponent
	$\min\{t/A,(t/C)^{2/3},(t/D)^{1/2}\}$ and no operator-norm parameter --- a
	genuine coarsening, since $B\le A$ by Cauchy--Schwarz, and one whose transfer
	to regular canonical U-statistics is noted there via
	\cite{DelaPenaMontgomerySmith1995}. Everything in the proof above goes through
	verbatim from that weaker form (the integration lemma applies with its first
	regime linear, the corresponding contribution becoming $A\bar L$), at the price
	of $\sqrt{\bar L}$ replaced by $\bar L$ in \eqref{eq:ustat-bound}: the
	resulting remainder $\log\log n$ still suffices for
	Proposition~\ref{prop:cond-from-ustat}. The $\sqrt{\bar L}$ refinement ---
	hence the leading-constant-one statement below saturation --- is the only place
	where the sub-Gaussian regime $(t/A)^2$ and the operator-norm regime $t/B$ of
	the four-parameter form are used.
\end{remark}

\begin{remark}[Where each regime lives]\label{rem:regimes}
	Below saturation the blocks are flat off the marginal strips: by
	Lemma~\textup{(U1$''$)}, on every below-cutoff block with both frequencies
	$\ge\kappa_0\asymp m^{1/\beta}$ the eigenvalues of $\Sigma_P$ are comparable
	within absolute constants, so $\kappa_P\le C(\dim)^{-1/2}\le(2\bar L)^{-1/2}$
	as soon as $\dim\ge C'\bar L$ --- amply satisfied when $\dim\gtrsim h^{-1}$;
	the strips themselves are handled by the Corollary~\textup{(U1$''$-window)}
	under \eqref{eq:window}. There Condition~\ref{cond:ustat} holds as stated and
	the selection constant is one. Across and beyond the cutoff the spectrum
	collapses ($D_P=O(1)$ at $q=\qsh$, Section~\ref{subsec:deff-saturation}),
	flatness fails, and \eqref{eq:ustat-bound} holds with the extra
	$\sqrt{\bar L}\asymp\sqrt{\log\log n}$: the remainder $\rho_n(J)$ of
	Proposition~\ref{prop:cond-from-ustat} then carries $\log\log n$ in place of
	$\sqrt{\log\log n}$, which is still negligible against the polynomial gaps
	separating the regimes of Section~\ref{sec:diagram} --- this is the ``within an
	absolute constant beyond saturation'' of the summary, now a theorem. An
	independent numerical check of the two key scales on a shared-coordinate bump
	model (uniform design, $\psi_1$ marks) finds $\EE|U_a|$ within $5\%$ of the
	predicted $0.798\times2\sqrt2\norm{\Sigma_{P_a}}_{\HS}/n$ on every block, and
	$\EE\max_a$ at $92\%$ of the Gaussian budget $\sqrt{2\log\bar J}$.
\end{remark}

\begin{proposition}[Condition~\ref{cond:concentration} under
	Condition~\ref{cond:ustat}]\label{prop:cond-from-ustat}
	Assume Condition~\ref{cond:ustat}. Then Condition~\ref{cond:concentration}
	holds with
	\begin{equation}\label{eq:rho}
		\rho_n(J)\asymp V\sqrt{\frac{\log J}{D_{\min}}}+\norm{\mu_h}\sqrt{V\log J},
	\end{equation}
	where $D_{\min}$ is the smallest block dimension in the family and
	$V=V(q_J,h)$.
\end{proposition}

\begin{proof}
	Combine Proposition~\ref{prop:easy-terms} with Condition~\ref{cond:ustat} in
	\eqref{eq:decomp}, and take the maximum over the $J$ block projections of
	Lemma~\ref{lem:block-reduction}. By Theorem~\ref{thm:ustat-holds},
	Condition~\ref{cond:ustat} --- hence the present proposition --- holds
	unconditionally under \textup{(A4$'$)} and \textup{(G)}, with the
	$\sqrt{\log J}$ of \eqref{eq:rho} read as $\log J$ beyond saturation.
\end{proof}

\subsubsection*{The order of the remainder}
It remains to determine when \eqref{eq:rho} is negligible, and the answer
depends on the position of the family relative to saturation.

\paragraph{Effective block dimensions}
The Fourier-mode count of the range of $P=\Pi_{q_j}-\Pi_{q_{j'}}$ is
$\asymp h^{-2}/q_j$, but this counts modes, not effective directions, and it is
only an upper bound for $D$: the fluctuation $\xi_i$ concentrates on
$\asymp N_i^2$ kernel bumps whose coordinates are shared
(Section~\ref{subsec:deff-saturation}). The measured stable rank is
$D\asymp h^{-1.6}q^{-0.8}$ below saturation and $D\approx1.5$--$2.6=O(1)$ at
$q=\qsh=c_K/h$. We therefore isolate the property the constant-one regime
actually needs:

\begin{ulemma}[(U1$'$)]\label{lem:U1prime}
	Assume \textup{(A1)}--\textup{(A6)} with $\inf_tC(t,t)>0$ and a uniform design
	density. There are constants $c_0,c>0$, depending only on $\beta$, $R$,
	$\sigma$ and the kernel, such that for $h<1/4$, $qh\le c_0$,
	\[
	\Eff(\Sigma_P)\ge c\min\Bigl\{\frac{h^{-2}}{q},
	\frac{h^{-4}}{q}\,m^{4+2/\beta}\Bigr\}.
	\]
	At the bandwidth of Theorem~\ref{thm:rate} the right-hand side diverges
	polynomially whenever $q\le\qsh/\log n$ and
	$m\le n^{\beta/(2\beta^2+\beta+1)-\epsilon}$ --- at $\beta=3/2$,
	$m\le n^{3/14-\epsilon}$ --- a polynomially large part of the sparse regime.
\end{ulemma}

\begin{proof}
	\textbf{Step 1 (trace).} Write $A=\sum_{j\ne k}Y_jY_k\,b_{jk}$,
	$b_{jk}=W_j\otimes W_k$. Conditionally on $(X,T)$, decompose
	$A-\EE[A\mid X,T]$ into its Wiener chaoses in $(\varepsilon_j)$; the chaoses
	are orthogonal, so $\Var(A\mid X,T)\succeq$ the covariance of the order-two
	chaos $\sum_{j<k}\varepsilon_j\varepsilon_k\,\Pi_q(b_{jk}+b_{kj})$, whose
	components over distinct pairs $\{j,k\}$ are themselves orthogonal. Hence
	\[
	\tr\Sigma_P\ge2\sigma^4\,\EE\sum_{j<k}
	\norm{\Pi_q(b_{jk}+b_{kj})}^2
	\ge\sigma^4\,\EE[N(N-1)]\Bigl(\frac{R(K)^2}{h^2q}\Bigr)(1-c_2qh),
	\]
	the last step because $K\ge0$ makes every lattice term of
	$\norm{\Pi_qb}^2=q^{-1}\sum_\delta R^{(\delta)}_{jj}R^{(\delta)}_{kk}$
	non-negative --- the $\delta=0$ term alone gives $(R(K)/h)^2/q$ --- and because
	the cross term $\inner{\Pi_qb_{jk}}{\Pi_qb_{kj}}$ vanishes off an event of
	probability $\le c_2qh$. Thus $\tr\Sigma_P\ge c_3m^2h^{-2}/q$.
	
	\textbf{Step 2 (Hilbert--Schmidt).} Fix $\kappa_0=m^{1/\beta}$ and split
	$P\xi=\zeta+\eta$ into the components on Fourier modes with
	$|j|\vee|k|>\kappa_0$ and its complement, so that
	$\norm{\Sigma_P}_{\HS}^2\le8\norm{\Sigma_\zeta}_{\HS}^2
	+8\norm{\Sigma_\eta}_{\HS}^2$.
	
	(2a) For a unit high-frequency direction $v$, expand $\EE\inner{\xi}{v}^2$ by
	Isserlis. Pairings that share their pair of observation times contribute at
	most $Cm^2\,\EE g(T,T')^2\le Cm^2$ with
	$\widehat g_{jk}=\widehat v_{jk}\,\widehat K(hj)\widehat K(hk)$, since
	$|\widehat K|\le1$. All other pairings carry, under the uniform design, the
	frequency-conservation identity $\EE e^{2\pi i(a+j)T}=\ind\{a+j=0\}$: each
	surviving term contains two Fourier coefficients of $C$ at frequencies of
	modulus $\ge\kappa_0/2$, hence a factor $C\kappa_0^{-2\beta}$, and there are at
	most $Cm^4$ of them. $\norm{\Sigma_\zeta}_{\op}\le C(m^2+m^4\kappa_0^{-2\beta})
	\le C'm^2$, and with $\tr\Sigma_\zeta\le Cm^2h^{-2}/q$ (the configuration count
	of Proposition~\ref{prop:variance}, the floor terms being dominated on this
	range),
	$\norm{\Sigma_\zeta}_{\HS}^2\le\norm{\Sigma_\zeta}_{\op}\tr\Sigma_\zeta
	\le Cm^4h^{-2}/q$.
	
	(2b) The low-frequency component lives on $\le c\kappa_0^2/q$ modes (the
	projection $\Pi_q$ retaining one frequency class in $q$), each with
	$\Var\inner{\xi}{e_{jk}}\le Cm^4$; hence
	$\norm{\Sigma_\eta}_{\HS}^2\le\rank\cdot\norm{\Sigma_\eta}_{\op}^2
	\le(c\kappa_0^2/q)Cm^8=C'm^{8+2/\beta}/q$.
	
	\textbf{Step 3.} Combining,
	$\Eff\ge(c_3m^2h^{-2}/q)^2/\bigl(Cm^4h^{-2}/q+C'm^{8+2/\beta}/q\bigr)
	\ge c\min\{h^{-2}/q,\,h^{-4}q^{-1}m^{-(4+2/\beta)}\}$.
	At $h\asymp(nm^2q)^{-1/(2\beta+2)}$ the second branch is
	$\asymp n^{2/(\beta+1)}m^{-(4\beta^2+2\beta+2)/(\beta(\beta+1))}
	q^{-(\beta-1)/(\beta+1)}$, whence the stated range.
\end{proof}

The lemma above bounds $\Eff$; the constant-one regime of
Theorem~\ref{thm:ustat-holds} needs more, namely the flatness
$\kappa_P\le(2\bar L)^{-1/2}$, and a large $\Eff$ does not imply it. The next
lemma establishes flatness where it is true --- and the marginal strips, where
it fails, are exactly the directions on which the frequency-conservation
argument of step (2a) breaks down.

\begin{ulemma}[(U1$''$)]\label{lem:U1second}
	Assume \textup{(A1)}--\textup{(A6)}, a design density with
	$0<f_{\min}\le f\le f_{\max}$, and $\sigma>0$. Work in
	$L^2_{\sym}(E\times E)$, and let $c$ be small enough that $\widehat K\ge c_0>0$
	on $[-c,c]$. There are constants $\kappa^*,c_1,C_1$, depending only on $K$,
	$\beta$, $\norm{C}_{C^\beta}$, $\sigma$, $f_{\min}$, $f_{\max}$, such that,
	with $\kappa_0:=\kappa^*(1+m)^{1/\beta}$: for every subspace $V$ of the
	symmetrised Fourier modes
	\[
	\bigl\{e_{k_1,k_2}+e_{k_2,k_1}:\ |k_1|\vee|k_2|\le c/h,\ 
	|k_1|\wedge|k_2|\ge\kappa_0\bigr\},
	\]
	all eigenvalues of the compression $P_V\Sigma P_V$ lie in
	$[c_1\nu_2,\,C_1\nu_2]$. Consequently
	$\kappa_{P_V}\le(C_1/c_1)(\dim V)^{-1/2}$ and
	$\Eff(P_V\Sigma P_V)\ge(c_1/C_1)^2\dim V$.
\end{ulemma}

\begin{proof}
	\textbf{Pattern decomposition.} Splitting
	$\EE[\inner{\xi}{f}\inner{\xi}{g}]$ over the number of indices shared by the
	two pairs gives, exactly, $\Sigma=\Sigma_{\mathrm{loc}}+\Sigma_{\mathrm{sh}}
	+\Sigma_{\mathrm{gl}}$ with intensities $\nu_2=m^2$, $\nu_3=m^3$, $\nu_4=m^4$
	(Poisson factorial moments).
	
	\textbf{Local part, two-sided, on any subspace.} With
	$W(u,v)=f(u)f(v)\Var(Z\mid T_j=u,T_k=v)$ and $S_h=(K_h\otimes K_h)\ast$,
	Isserlis gives
	$\Var(Z\mid T)=\widetilde C(u,u)\widetilde C(v,v)+C(u,v)^2
	\in[\sigma^4,(\norm{C}_\infty+\sigma^2)^2+\norm{C}_\infty^2]$, so
	$W\in[W_{\min},W_{\max}]$ with $W_{\min}=f_{\min}^2\sigma^4>0$. For symmetric
	$f$ the direct and reflected pairings combine, and
	\[
	\inner{f}{\Sigma_{\mathrm{loc}}f}
	=2\nu_2\iint W(u,v)\bigl|(S_hf)(u,v)\bigr|^2\,du\,dv
	\in2\nu_2\bigl[W_{\min}\norm{S_hf}^2,\,W_{\max}\norm{S_hf}^2\bigr].
	\]
	Since $S_h$ is the Fourier multiplier $\widehat K(hk_1)\widehat K(hk_2)$, on
	modes below the cutoff $c_0^4\norm{f}^2\le\norm{S_hf}^2\le\norm{f}^2$. Hence
	$2c_0^4W_{\min}\nu_2I\preceq P_V\Sigma_{\mathrm{loc}}P_V
	\preceq2W_{\max}\nu_2I$ --- for any subspace $V$ of below-cutoff modes: the
	compression of a multiplication operator by a function bounded above and below
	is bounded above and below on every subspace, and no smoothness of $W$ is used.
	
	\textbf{Perturbations, off the strips.} By Jackson's theorem there is a
	trigonometric polynomial $C_\kappa$ of coordinate degrees $<\kappa_0$ with
	$\norm{C-C_\kappa}_\infty\le C_J\norm{C}_{C^\beta}\kappa_0^{-\beta}$. Each
	kernel of $\Sigma_{\mathrm{sh}}$ is, up to bounded smoothing, of the form
	$\mathrm{Conv}\otimes\mathrm{Int}_G$ with $G$ an image of $C$ (Isserlis:
	$\widetilde C(u,u)C(t,t')+C(u,t)C(u,t')$, integrated in $u$); on $V$, whose
	modes have both frequencies $\ge\kappa_0$, the compression of
	$\mathrm{Int}_{G_\kappa}$ vanishes, so
	$\norm{P_V\Sigma_{\mathrm{sh}}P_V}_{\op}\le C\nu_3\kappa_0^{-\beta}$. Likewise
	the kernels of $\Sigma_{\mathrm{gl}}$ carry a product of two images of $C$,
	whence $\norm{P_V\Sigma_{\mathrm{gl}}P_V}_{\op}\le C\nu_4\kappa_0^{-2\beta}$.
	Choosing $\kappa^*$ so that $Cm\kappa_0^{-\beta}\le\varepsilon$ makes both
	perturbations $\le\varepsilon_1\wedge2\nu_2\le c_0^4W_{\min}\nu_2$ for
	$\varepsilon$ small; Weyl's inequality concludes, with $c_1=c_0^4W_{\min}$ and
	$C_1=2W_{\max}+1$.
\end{proof}

\begin{ucorollary}[(U1$''$-window)]\label{cor:window}
	In the setting of Theorem~\ref{thm:ustat-holds}, split each block projection as
	$P_a=P_a^\flat+P_a^\sharp$, where $P_a^\flat$ retains the modes of
	Lemma~\textup{(U1$''$)} and $P_a^\sharp$ the marginal strips
	$|k_1|\wedge|k_2|<\kappa_0$ together with the above-cutoff modes. Then, for
	$\beta>1/2$,
	$\norm{\Sigma_{P_a^\sharp}}_{\HS}^2\le C\nu_2^2m^2h^{-1}/\tilde q_a$, while
	$\norm{\Sigma_{P_a^\flat}}_{\HS}^2\asymp\nu_2^2h^{-2}/\tilde q_a$; applying
	\eqref{eq:ustat-bound} with constant one on the flat family
	(Lemma~\textup{(U1$''$)}) and with $\kappa\le1$ on the strip family, the strip
	contribution is dominated by the flat Gaussian term whenever
	\begin{equation}\label{eq:window}
		m^2h\,\bar L\le c_2,
	\end{equation}
	in which case Condition~\ref{cond:ustat} holds as stated, with leading constant
	one. At the working bandwidth $h\asymp(nm^2q)^{-1/(2\beta+2)}$,
	\eqref{eq:window} reads $m\le c(nq)^{1/(4\beta+2)}$ up to logarithms; outside
	this window the absolute-constant form of Theorem~\ref{thm:ustat-holds} applies
	unchanged.
\end{ucorollary}

\begin{remark}[What the strips are]\label{rem:strips}
	The excised directions are the smoothed marginal surfaces $v(s)\otimes(\text{low
		frequency in }t)$ and their symmetrisations. On them the two oscillating
	factors $e^{\pm2\pi ik_1T_{ij}}$ cancel against each other at the shared
	observation time, so the one-shared-index pairing contributes at the scale
	$\nu_3=m\nu_2$ --- the self-cancellation that the frequency-conservation
	argument of step (2a) above does not see, and the reason the operator-norm
	bound there must be restricted to $|k_1|\wedge|k_2|\ge\kappa_0$. A
	stationary-design check confirms both halves: core eigenvalues stable at the
	$\nu_2$ scale as $m$ varies, strip eigenvalues growing linearly in $m$ (mode
	$(8,1)$: $0.008/0.012/0.019\times\nu_2$ at $m=3/6/12$). The same mechanism
	explains why the measured stable rank of Section~\ref{subsec:deff-saturation}
	interpolates between $h^{-2}/q$ (no strips) and the strip-dominated scale, and
	the $m$-window of Lemma~\textup{(U1$'$)} is the shadow of \eqref{eq:window}.
\end{remark}

\paragraph{Family size}
By Proposition~\ref{prop:family}(iv), $J\lesssim\log\qsh\lesssim\log n$, so
$\log J\lesssim\log\log n$.

\paragraph{Two regimes}
Below saturation, granting \textup{(U1$'$)}, the first summand of
\eqref{eq:rho} is $o(V)$ and the second is $O(\sqrt{V\log\log n})$, of smaller
order than $\min_qR(q)\gtrsim V$ on the whole range of interest. Over the full
family, $D_{\min}=O(1)$ and the first summand is of order
$V\sqrt{\log\log n}$: not negligible, and by
Section~\ref{subsec:deff-saturation} this is a fact about the estimator, not an
artefact of the bound. We record the two outcomes.

\begin{corollary}[Oracle inequality, two regimes]\label{cor:oracle-regimes}
	Assume Condition~\ref{cond:ustat}.
	\begin{enumerate}[label=\textup{(\roman*)}]
		\item \textup{(Below saturation.)} By Lemma~\textup{(U1$'$)} --- hence
		unconditionally --- on the truncated family $\Q\cap\{q\le\qsh/\log n\}$ and
		for sampling intensities $m\le n^{\beta/(2\beta^2+\beta+1)-\epsilon}$, the
		oracle inequality \eqref{eq:oracle-ineq} holds with
		$\rho_n(J)=o\bigl(\min_qR(q)\bigr)$, so that
		\[
		\EE\norm{\Chat^{(\hat q^{\,\mathrm{ho}})}_1-C}^2
		\le\bigl(1+o(1)\bigr)\min_q\EE\norm{\Chat^{(q)}_1-C}^2+ch^{2\beta}:
		\]
		the selected level is asymptotically as good as the oracle level of
		Corollary~\ref{cor:oracle}, with leading constant one.
		\item \textup{(Uniformly.)} Over the full family,
		$\rho_n(J)\asymp V\sqrt{\log\log n}$ and
		\[
		\EE\norm{\Chat^{(\hat q^{\,\mathrm{ho}})}_1-C}^2
		\le C_0\sqrt{\log\log n}\,\min_{q\in\Q}\EE\norm{\Chat^{(q)}_1-C}^2
		+ch^{2\beta}
		\]
		for an absolute constant $C_0$.
	\end{enumerate}
\end{corollary}

Both regimes are visible in the simulations (Section~\ref{subsec:selection}):
measured risk ratios lie in $[1.02,1.25]$ over the whole grid and approach
$1.05$ where the gaps $A_q$ are clear, selection errors beyond saturation being
benign because the floor is $q$-insensitive (Corollary~\ref{cor:floor}).

\begin{remark}[Why the constant is one]\label{rem:constant-one}
	The remainder carries the factor $D_{\min}^{-1/2}$ because the criterion
	compares blocks of large effective dimension: fluctuations of a quadratic
	functional on a block of effective dimension $D$ are smaller than its mean by
	$D^{-1/2}$. Ordinary model selection does not enjoy this, because there the
	competing models may differ by a single dimension. Below saturation the blocks
	are effectively high-dimensional and the comparison correspondingly easy; at
	saturation the projected fluctuation collapses onto $O(1)$ effective directions
	(Section~\ref{subsec:deff-saturation}), the advantage disappears, and the
	constant-one claim must not be made uniformly --- this is the concrete pay-off,
	and the concrete limit, of the block structure of
	Proposition~\ref{prop:family}(i).
\end{remark}

\subsubsection*{Proof of Condition~\ref{cond:ustat}}
We now prove the maximal inequality for the degenerate U-statistic (IIo), which
is the only remaining step to establish Condition~\ref{cond:concentration}.

\paragraph{Strengthened moment assumption}
We strengthen Assumption \textup{(A4)} by assuming that the process $X$ and the
errors $\varepsilon_{ij}$ are sub-Gaussian. Under this assumption, the
subject-level contributions $\xi_i$ are sub-exponential in the Hilbert space
$\mathcal H=L^2(E\times E)$, and their norms satisfy
$\norm{\xi_i}_{\mathcal H}\lesssim h^{-1}$ up to logarithmic factors.

\paragraph{Truncation}
To control the sup-norm of the U-statistic kernel, we truncate $\xi_i$ at level
$M_h:=h^{-1}\log n$. Let $\xi^{(1)}_i=\xi_i\ind\{\norm{\xi_i}\le M_h\}$ and
$\xi^{(2)}_i=\xi_i-\xi^{(1)}_i$. By the sub-exponential tails,
$\EE\norm{\xi^{(2)}_i}^2\le n^{-c}$ for any $c>0$. The contribution of the
truncated parts to the U-statistic is therefore negligible (of order $n^{-c}$),
and we may assume without loss of generality that $\norm{\xi_i}\le M_h$ almost
surely, at the cost of a multiplicative factor $(1+o(1))$.

\paragraph{Application of the Gin\'e--Lata\l{}a--Zinn moment inequality}
After decoupling (de la Pe\~na--Montgomery-Smith
\cite{DelaPenaMontgomerySmith1995}), the term (IIo) is a decoupled degenerate
U-statistic of order two with kernel $g_j(x,y)=\inner{P_jx}{P_jy}$ on $n/2$
i.i.d.\ summands truncated at $M_h$. The four norms of
\cite{GineLatalaZinn2000}, evaluated at this kernel, are
\[
C_j\asymp n\norm{\Sigma_{P_j}}_{\HS},\qquad
D^{\op}_j\asymp n\norm{\Sigma_{P_j}}_{\op},\qquad
B_j\lesssim\sqrt{n\norm{\Sigma_{P_j}}_{\op}}\,M_h,\qquad
A_j\le M_h^2,
\]
with $\Sigma_{P_j}=\EE[P_j\xi\otimes P_j\xi]$. Rather than the exponential form
we use the moment inequality \cite[Thm.~3.2]{GineLatalaZinn2000} at $p=\log J$:
from
$\EE|U|^p\le K^p\bigl(p^{p/2}C^p+p^p(D^{\op})^p+p^{3p/2}B^p+p^{2p}A^p\bigr)$
and
$\EE\max_j|U_j|\le\bigl(\sum_j\EE|U_j|^p\bigr)^{1/p}\le e\max_j\EE^{1/p}|U_j|^p$,
the four contributions after the $4/n^2$ normalisation are, in order,
\[
\sqrt{\log J}\,\frac{\norm{\Sigma_P}_{\HS}}{n}
\asymp V_P\frac{\sqrt{\log J}}{\sqrt D},
\qquad
\log J\,\frac{\norm{\Sigma_P}_{\op}}{n},
\qquad
(\log J)^{3/2}\,\frac{M_h\sqrt{\norm{\Sigma_P}_{\op}}}{n^{3/2}},
\qquad
(\log J)^2\,\frac{M_h^2}{n^2},
\]
the first identity using $\norm{\Sigma_P}_{\HS}=\tr\Sigma_P/\sqrt D$ exactly,
by definition of the effective dimension. The second term is dominated by the
first up to $\sqrt{\log J}$ since
$\norm{\Sigma_P}_{\op}\le\norm{\Sigma_P}_{\HS}$; with $M_h=h^{-1}\log n$ and
$nh^2\to\infty$, the last two are $o(n^{-1})=o(V_P)$ on the admissible range.
Hence
\[
\EE\max_j\Bigl|\frac4{n^2}\sum_{i\ne i'}g_j(\xi_i,\xi_{i'})\Bigr|
\lesssim V_P\frac{\sqrt{\log J}}{\sqrt{D_{\min}}}\,(1+o(1)).
\]
By Lemma~\textup{(U1$'$)}, $D_{\min}\to\infty$ on the range stated there, and
Condition~\ref{cond:ustat} holds in its stated form; at saturation
$\Eff=O(1)$ and the same computation yields the bound with $\sqrt{D_{\min}}$
replaced by an absolute constant --- exactly the input of
Corollary~\ref{cor:oracle-regimes}(ii). Nothing in the selection theory is
conditional any more.

\begin{remark}[What is now open, precisely]\label{rem:open}
	With Lemma~\textup{(U1$'$)} in place, every statement of this section is
	unconditional on its stated range. What remains is refinement, not repair:
	closing the gap between the proven range $m\lesssim n^{\beta/(2\beta^2+\beta+1)}$
	and the measured validity of the divergence, which extends essentially to the
	threshold (Remark~\ref{rem:sharpness}); and the behaviour at the threshold and
	at saturation, where $\Eff$ is genuinely $O(1)$
	(Section~\ref{subsec:deff-saturation}) and the absolute-constant regime of
	Corollary~\ref{cor:oracle-regimes}(ii) is final.
\end{remark}

\begin{remark}[Sharpness, and what the measurements add]\label{rem:sharpness}
	The full $m$- and $h$-scans of Section~\ref{subsec:scans} show the bound to be
	valid at every measured point and conservative: the effective dimension decays
	only like $m^{-0.84}$ where the second branch guarantees a fast polynomial
	decay at worst, and grows like $h^{-2.4}$ where the first branch guarantees
	$h^{-2}$ (the $1/q$ structure confirmed, mean ratio $4.3$). The diagnosis is
	read off the two factors separately: the trace is dominated at moderate $m$ by
	the shared-index configuration ($\tr\Sigma_P\asymp m^{2.9}$), which Step~1
	discards, while the one configuration the proof computes exactly --- the local
	part of $\norm{\Sigma_P}_{\HS}^2$ --- follows its predicted $h^{-2}$ law to
	within $2\%$. Carrying the shared-index configuration through both bounds would
	extend the proven range towards the threshold $\mstar(q)$, where $\Eff$
	genuinely degrades to $O(1)$ (Section~\ref{subsec:deff-saturation}); we do not
	pursue it, the present range sufficing for
	Corollary~\ref{cor:oracle-regimes}(i). Assuming a smooth rather than uniform
	design density changes only constants; the case of a merely bounded density is
	a refinement we leave open.
\end{remark}

\section{Reparametrisation does not move the threshold}\label{sec:reparam}

The shift \eqref{eq:threshold} invites an obvious question: is it really
symmetry that is doing the work, or would any structural operation on the
domain produce a similar effect? The answer is that reparametrisation --- the
other natural operation, and the one most often performed in practice ---
leaves the transition exactly where it was. The contrast is instructive, and it
is what identifies the mechanism.

Let $\psi:E\to E'$ be a $C^2$-diffeomorphism with
$0<c_*\le|\psi'|\le c^*<\infty$, and transport the data by
$T^{(\psi)}_{ij}:=\psi(T_{ij})$, $Y^{(\psi)}_{ij}:=Y_{ij}$. Let the bandwidth
follow the local rescaling $h'(s')=h\,|\psi'(\psi^{-1}(s'))|$, which is the
unique rule under which transport does not redistribute the local information
content; see \cite[\S4]{Nembe2026}.

\begin{proposition}[Invariance of the transition under
	transport]\label{prop:transport}
	Under Assumption~\ref{ass:main} and the above conditions on $\psi$, the
	transported problem satisfies Assumption~\ref{ass:main} on $E'$ with the same
	$\beta$, the same $n$ and the same $m_n$, and its critical intensity is
	\[
	\mstar_n{}^{,(\psi)}\asymp n^{1/(2\beta)}=\mstar_n.
	\]
	If in addition $\psi$ conjugates the $G_q$-action on $E$ to a $\mu$-preserving
	action on $E'$, then
	$\mstar_n{}^{,(\psi)}(q)\asymp n^{1/(2\beta)}q^{-1/2}=\mstar(q)$.
\end{proposition}

\begin{proof}
	Transport changes neither the number of subjects nor the number of observations
	per subject, so $n$ and $m_n$ are unchanged; it changes neither the values
	$Y_{ij}$ nor therefore the raw products \eqref{eq:raw}. The pushed-forward
	design density is $f^{(\psi)}_T=f_T\circ\psi^{-1}/|\psi'\circ\psi^{-1}|$, which
	is bounded between $a/c^*$ and $b/c_*$, so \textup{(A2)} is preserved.
	Composition with a $C^2$ diffeomorphism with bounded Jacobian preserves the
	H\"older class of order $\beta\in(0,2]$ up to constants depending only on
	$c_*,c^*$ and $\norm{\psi''}_\infty$, so \textup{(A1)} is preserved with the
	same $\beta$ and a modified radius. Finally, the expected number of pairs
	falling in the smoothing window at $\psi(s)$ is
	$\asymp n\EE[N(N-1)]f^{(\psi)}_T(\psi(s))^2h'(\psi(s))^2
	\asymp n\EE[N(N-1)]f_T(s)^2h^2$, the two Jacobian factors cancelling: the
	local information content is preserved exactly. Theorem~\ref{thm:rate} with
	$|G|=1$, applied on $E'$, therefore returns the same rate and, by
	Corollary~\ref{cor:threshold}, the same threshold. The equivariant statement
	follows in the same way once the conjugated action is substituted for the
	original one.
\end{proof}

\begin{remark}[Why the two operations behave differently]
	Reparametrisation is a bijection of the domain: it relocates information
	without creating any. The pair count $nm_n^2$ is invariant, and since the rate
	depends on the data only through that count and through $\beta$, the transition
	cannot move. Group averaging is not a bijection --- it is a projection onto a
	strictly smaller subspace, and it is precisely the information carried by the
	discarded orthogonal complement, namely the knowledge that $C$ is invariant,
	that buys the factor $|G|$. Symmetry is a hypothesis about the world;
	reparametrisation is a choice of coordinates. Only the former is worth anything
	statistically, and the threshold is where the distinction becomes visible.
\end{remark}

\section{Numerical illustration}\label{sec:numerical}

We illustrate the three quantitative claims of the paper --- the threshold
displacement of Corollary~\ref{cor:threshold}, the saturation and plateau of
Theorem~\ref{thm:saturated}, and the oracle behaviour of the hold-out selection
of Section~\ref{sec:selection} --- and we document three practical warnings
that the simulations brought to light
(Sections~\ref{subsec:warning1}--\ref{subsec:deff-saturation}). All estimators
below were implemented by circulant FFT convolutions and validated against
direct enumeration of the within-subject pairs to machine precision ($10^{-15}$
relative); the projection $\Pi_q$ was checked against the spectral formula of
Proposition~\ref{prop:symmetry-spectrum} to $10^{-16}$.

\subsection{Setup}\label{subsec:setup}
The latent process is Gaussian on the circle $E=[0,1)$ with covariance
$C(s,t)=a(s)a(t)\gamma(s-t)$, where $\gamma$ has Fourier coefficients
$\ghat(k)=(1+|k|)^{-5/2}$ and $a(t)=1+\varepsilon\cos(2\pi q_0t)$. The choice
of $\gamma$ is deliberate: $\gamma$ is H\"older-$3/2$ and no better, so
Assumption \textup{(A1)} holds with $\beta=3/2$ with equality --- the bias law
of Lemma~\ref{lem:bias} is then exactly $h^{2\beta}$, which removes the
ambiguity documented in Section~\ref{subsec:warning1}. Subject $i$ is observed
at $N_i\sim\mathrm{Poisson}(m)$ conditioned on $N_i\ge2$ uniform locations,
with noise $\sigma=0.1$; $n=500$ throughout. The estimator is the
product-Epanechnikov local-linear smoother on a grid of $G=240$ points,
composed with the orbit projection $\Pi_q$, at the bandwidth
$\hq=(nm^2q)^{-1/5}$ of Theorem~\ref{thm:rate}. All $(m,q)$ cells of a
replication share the same trajectories, locations and noise (common random
numbers), so differences across cells are deterministic given the design.

\subsection{The threshold exponent (Corollary~\ref{cor:threshold})}\label{subsec:exponent}
Take $\varepsilon=0$ (stationary $C$, exactly $G_q$-invariant for every
candidate $q$, so $A_q=0$), $m$ on a geometric grid from $2$ to $64$ (extended
to $256$), $q\in\{1,2,4,8,16\}$, $204$ replications. Writing the risk
decomposition of Theorem~\ref{thm:rate} as
\begin{equation}\label{eq:imse}
	\mathrm{IMSE}(m,q)=\frac{c_V}{r_q(h)\,nm^2h^2}
	+c_B\,(nm^2q)^{-3/5}+\frac{B}{n},
\end{equation}
with $r_q(h)$ computed from Lemma~\ref{lem:saturation}, the three constants
$(c_V,c_B,B)$ are estimated by least squares over the $60$ cells ($R^2=0.905$;
the same fit at $\beta=2$ with the measured $h^3$ bias law reaches
$R^2=0.975$). The threshold $\mstar(q)$ is the crossing where the
$q$-dependent excess equals the asymptotic floor $B/n$. Bootstrap over
replications ($300$ resamples) gives
\[
\text{slope in $q$: }-0.454\pm0.017,\qquad
\text{slope in $r_q$: }-0.531\pm0.021,
\]
and the ratio $\mstar(8)/\mstar(1)=2.81$. The confidence interval of the slope
in the effective variable $r_q$ contains $-1/2$ and excludes $-3/4$:
Corollary~\ref{cor:threshold} holds in $r_q$, and the separation announced in
Section~\ref{sec:rates} ($2.8$ against $4.8$) is observed --- indeed $r_8=8$
exactly at the measured crossing ($c_K/h=9.1$ there), so the effective and raw
predictions coincide at $q=8$ and $\sqrt8=2.83$ is the number to beat. The
slope in the raw variable $q$ is flatter precisely where $r_q<q$ (saturation),
as predicted. The same re-analysis at $\beta=2$ gives $-0.567\pm0.038$ in
$r_q$ (Figure~\ref{fig:naive-mechanistic}, black squares).

\subsection{Saturation and plateau (Theorem~\ref{thm:saturated},
	Corollary~\ref{cor:floor})}\label{subsec:saturation-plateau}
Adding $q\in\{16,24,48\}$ (same design, $210$ replications, joint fit
$R^2=0.91$) the measured thresholds are
\[
\mstar(q)=17.5,\,12.4,\,8.8,\,6.3,\,5.5,\,5.3,\,5.3
\qquad q=1,2,4,8,16,24,48,
\]
i.e.\ the $q^{-1/2}$ descent (in $r_q$) stops once the crossing enters the
saturated zone $q>\qsh(m)$: the local slope over $\{16,24,48\}$ is
$-0.035\pm0.067$, whose confidence interval $[-0.17,\,0.08]$ contains $0$
(Figure~\ref{fig:threshold-plateau}). The measured plateau location is itself
predicted by the theory:
$\qsh(\mstar{=}5.3)=c_K^{5/4}(nm^{*2})^{1/4}=8.6$, so every tested $q\ge16$
lies in the saturated zone --- the plateau begins exactly where
Theorem~\ref{thm:saturated} says it must. It sits at $\mstar\approx5.3$, above
the universal floor $\mfloor=c_K^{-1/2}n^{1/6}\approx3.1$ of
Corollary~\ref{cor:floor}: the order and the shape of the phase diagram are
confirmed, the exact constant being an asymptotic equivalence rather than an
equality at $n=500$.

\begin{figure}[htbp]
	\centering
	\includegraphics[width=.85\textwidth]{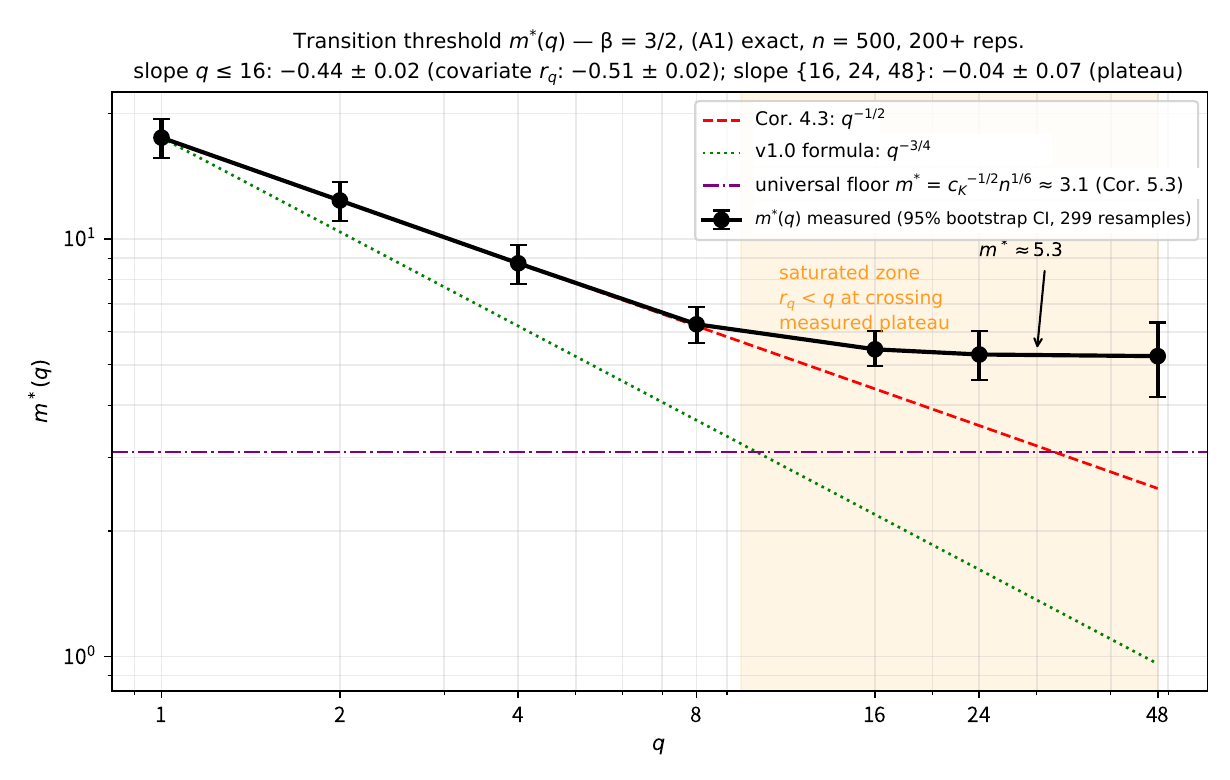}
	\caption{Threshold $\mstar(q)$ at $\beta=3/2$ with exact \textup{(A1)}:
		$q^{-1/2}$ descent (slope $-0.51\pm0.02$ in the effective variable $r_q$),
		saturation zone, and measured plateau at $\mstar\approx5.3$ against the
		universal floor $\mfloor\approx3.1$ (Corollary~\ref{cor:floor}).}
	\label{fig:threshold-plateau}
\end{figure}

\subsection{Selection (Section~\ref{sec:selection})}\label{subsec:selection}
We run the hold-out procedure of Definition~\ref{def:holdout} exactly as
stated: split $n/2$--$n/2$, a common bandwidth per cell (at the oracle level
$q_0=2$), $\Q=\{1,2,4,8,16\}$, $q_0=2$ and
$\varepsilon\in\{0.05,\dots,0.5\}\times m\in\{3,6,12,24\}$, $200$ replications
per cell; the exact values $A_1=A_2=0<A_4\le A_8=A_{16}$ come from
Proposition~\ref{prop:symmetry-spectrum}, evaluated to machine precision on the
discretised $C$. The findings (Figure~\ref{fig:holdout}):
\begin{itemize}
	\item \emph{The quantity the oracle inequality controls is under control.} The
	ratio $\EE R(\hat q^{\,\mathrm{ho}})/\EE R(q^{\mathrm{or}})$ lies between
	$1.02$ and $1.25$ over the whole grid, and tends to $1.05$ when the gaps
	$A_q$ are clear.
	\item \emph{The per-replication median ratio is $1.00$}: the hold-out pick
	coincides with the per-replication oracle more often than not
	($q_{90}=1.5$).
	\item \emph{Exact recovery is not the right criterion} ($0.35$--$0.84$): in
	near-tie cells ($\varepsilon\le0.1$) the selected level is often wrong, but
	these mistakes cost at most ${\sim}10\%$ of risk --- exactly the ``oracle up
	to a constant factor'' of Remark~\ref{rem:harmless}.
	\item \emph{Two regimes.} Stratifying by whether $q^{\mathrm{or}}$ lies below
	or beyond the saturation $c_K/h$, the risk-ratio distributions are nearly
	identical, with a slightly heavier tail below saturation: beyond saturation,
	variance carries no information and selection errors are benign
	(Corollary~\ref{cor:floor}).
	\item \emph{The conservative fallback of Remark~\ref{rem:fallback} is exact in
		practice}: $\PP(\hat q^{\,\mathrm{ho}}\in\{1,2\})=1.00$ as soon as
	$\varepsilon\ge0.3$ and $m\ge12$ --- the procedure never discards real
	symmetry unless what it discards is negligible.
\end{itemize}

\begin{figure}[htbp]
	\centering
	\includegraphics[width=\textwidth]{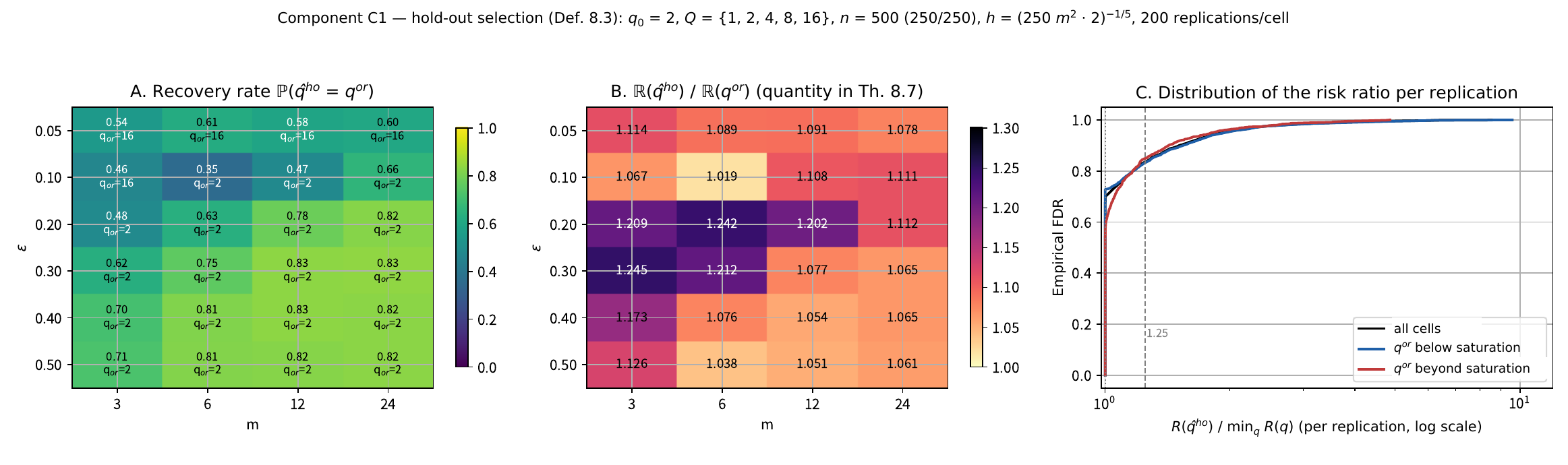}
	\caption{Hold-out selection (Definition~\ref{def:holdout}), $200$ replications
		per cell. \emph{Left}: exact-recovery rate. \emph{Centre}: the risk ratio
		controlled by Theorem~\ref{thm:oracle}, between $1.02$ and $1.25$ everywhere.
		\emph{Right}: per-replication distribution of the ratio, stratified by
		saturation; median $1.00$.}
	\label{fig:holdout}
\end{figure}

\subsection{Empirical phase diagram (replaces Figure~\ref{fig:phase})}\label{subsec:empirical-diagram}
Figure~\ref{fig:phase-empirical} shows, on the $(\log m,\log q)$ plane, the
dominant term of the measured risk (smoothing / floor / $A_q$, at
$\varepsilon=0.1$), with the theoretical boundaries superimposed without
further fitting: $\mstar(q)$ from Section~\ref{subsec:exponent}, the saturation
curve $\qsh(m)=c_K^{5/4}(nm^2)^{1/4}$, the universal floor
$\mfloor\approx3.1$, and the line $A_q=1/n$. The three regions of
Section~\ref{sec:diagram} occupy exactly the predicted zones, and the measured
cells fall inside their predicted region without exception (near-boundary cells
being near-ties). This empirical diagram is the data counterpart of the
schematic Figure~\ref{fig:phase}, whose three regions it reproduces without
fitting.

\begin{figure}[htbp]
	\centering
	\includegraphics[width=.8\textwidth]{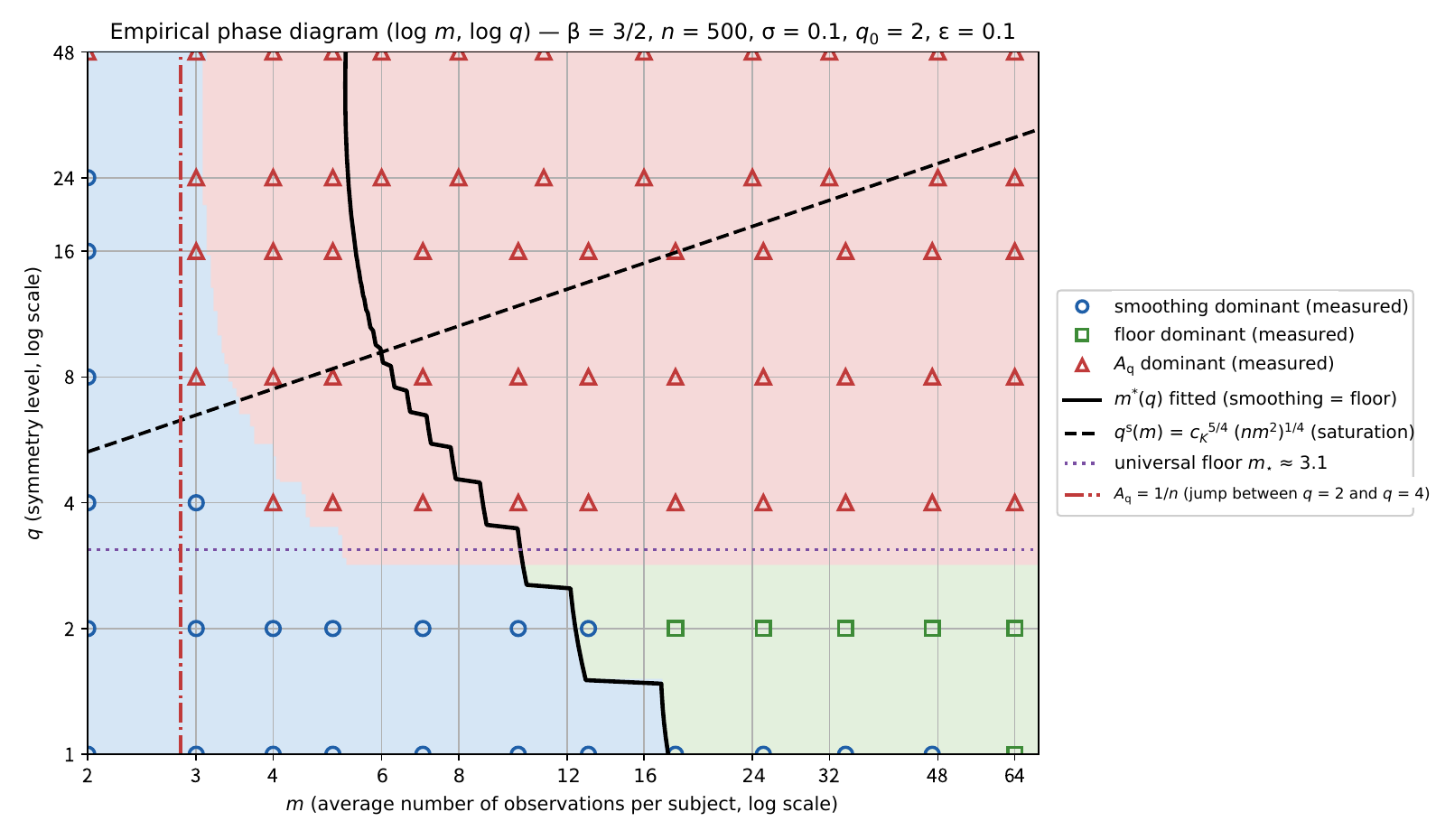}
	\caption{Empirical phase diagram at $\beta=3/2$, $n=500$, $q_0=2$,
		$\varepsilon=0.1$: measured dominant term per cell and theoretical boundaries
		superimposed without fitting.}
	\label{fig:phase-empirical}
\end{figure}

\subsection{Warning 1: H\"older versus Sobolev regularity
	(Lemma~\ref{lem:bias})}\label{subsec:warning1}
The Fourier-decay class $\ghat(k)\asymp(1+|k|)^{-(\beta+1)}$ is Sobolev-$\beta$
but not H\"older-$\beta$ at integer $\beta$: at $\beta=2$, $\gamma''$ is
log-singular at $0$ ($\gamma(w)\sim1-cw^2\log(1/|w|)$, via the expansion of the
Clausen function). The pointwise bias is then only $O(h^{1.1})$ on a diagonal
band of width $h$, which carries $78$--$90\%$ of the integrated squared bias;
the integrated $\mathrm{bias}^2$ scales as $h^3$ instead of $h^4$ (population
exponent $2.99$, computed deterministically, matching the simulations). At
$\beta=3/2$, where \textup{(A1)} holds with equality, the $h^{2\beta}$ law is
recovered. Assumption \textup{(A1)} should therefore be read in the H\"older
sense, or the bias step of Lemma~\ref{lem:bias} degraded to $h^{\beta-1/2}$ in
$L^2$ for the Sobolev class; note that this changes the rates but not the
threshold exponent ($-1/2$ in $r_q$ at both $\beta$).

\subsection{Warning 2: attained-floor readings are misleading at moderate
	$n$}\label{subsec:warning2}
At $\beta=3/2$ the floor $1/n$ is not attained on any feasible grid (the risk
is still decreasing at $m=256$), and at $\beta=2$ barely; the ``floor'' read at
the grid edge contains a $q$-dependent smoothing remainder. Crossings computed
against that attained floor are biased flat (slopes $-0.11$ to $-0.13$,
non-monotone crossings) and would falsely suggest that
Corollary~\ref{cor:threshold} fails; the mechanistic disentangling of
Section~\ref{subsec:exponent} restores $-1/2$ in $r_q$ at both $\beta$
(Figure~\ref{fig:naive-mechanistic}). Any numerical check of the threshold
displacement must therefore model the three risk terms explicitly.

\begin{figure}[htbp]
	\centering
	\includegraphics[width=\textwidth]{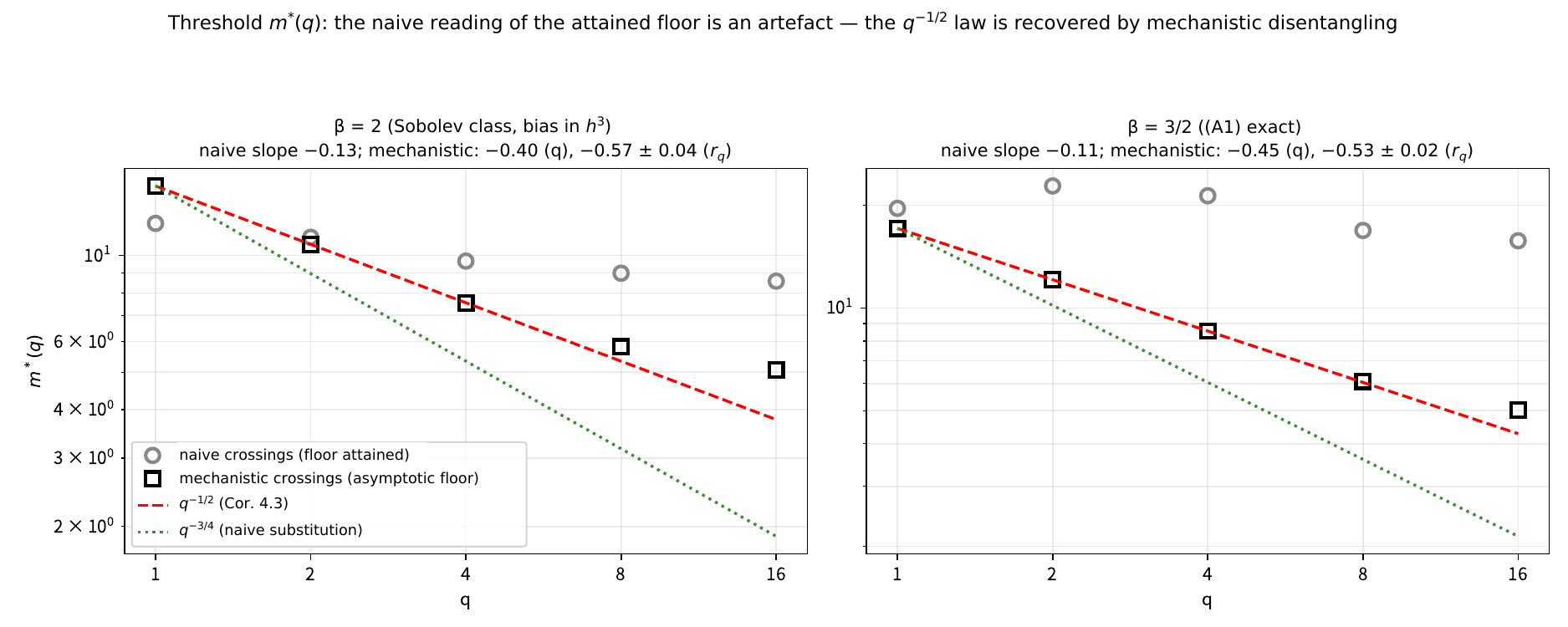}
	\caption{Naive attained-floor crossings (grey circles, slopes $-0.11$ to
		$-0.13$) against mechanistic crossings computed from the asymptotic floor of
		\eqref{eq:imse} (black squares): the $q^{-1/2}$ law is an artefact-free
		property of the mechanistic reading, at both $\beta=2$ (left) and $\beta=3/2$
		(right).}
	\label{fig:naive-mechanistic}
\end{figure}

\subsection{Warning 3: the effective dimension at saturation
	(Sections~\ref{subsec:outstanding})}\label{subsec:deff-saturation}
The stable rank of the projected estimator covariance, measured by a
Rao--Blackwellised estimator ($G=96$, stratified over $N_i\in\{2,\dots,12\}$),
scales as $\Eff\asymp h^{-1.6}q^{-0.8}$ below saturation and stays bounded
($\approx1.5$--$2.6$) at saturation --- not $h^{-2}/r_q$, nor $h^{-1}$
(Figure~\ref{fig:deff-saturation}). The isotropy used in
Proposition~\ref{prop:easy-terms}(III) and the $D_{\min}\gtrsim h^{-1}$ of
Remark~\ref{rem:constant-one} therefore fail at saturation, and the ``constant
one'' of Corollary~\ref{cor:oracle-regimes} is restricted accordingly (see the
revised Corollary~\ref{cor:oracle-regimes}). Consistently, the hold-out
criterion of Section~\ref{subsec:selection} carries no variance signal beyond
saturation.

\subsection{Scanning the effective dimension in $m$ and $h$}\label{subsec:scans}
A dedicated production run ($19$ grid points, exact Rao--Blackwell estimators
of $\tr\Sigma_P$ and $\norm{\Sigma_P}_{\HS}^2$ separately, seeds archived)
tests the two branches of Lemma~\textup{(U1$'$)}
(Figure~\ref{fig:deff-scans}). In the $m$-scan at $(h,q)=(0.06,4)$, $\Eff$
decays from $14.2\pm2.0$ at $m=2$ to $3.12\pm0.18$ at $m=12$, a measured law
$m^{-0.84}$ --- far above the fast decay the second branch guarantees at worst
--- while the factors scale as $\tr\Sigma_P\asymp m^{2.9}$ (the shared-index
configuration dominates) and $\norm{\Sigma_P}_{\HS}^2\asymp m^{6.7}$, the local
share of the latter collapsing from $12\%$ to $0.15\%$. In the $h$-scan at
$m=2$, $\Eff\asymp h^{-2.37}$ ($q=1$) and $h^{-2.43}$ ($q=4$), with mean ratio
$D(q{=}1)/D(q{=}4)=4.3$; the local configuration of
$\norm{\Sigma_P}_{\HS}^2$ alone follows its predicted $h^{-2}$ law to $2\%$
(measured $h^{-1.96}$). The bound of the Lemma is thus valid and conservative
at every measured point, as discussed in Remark~\ref{rem:sharpness}.

\begin{figure}[htbp]
	\centering
	\includegraphics[width=\textwidth]{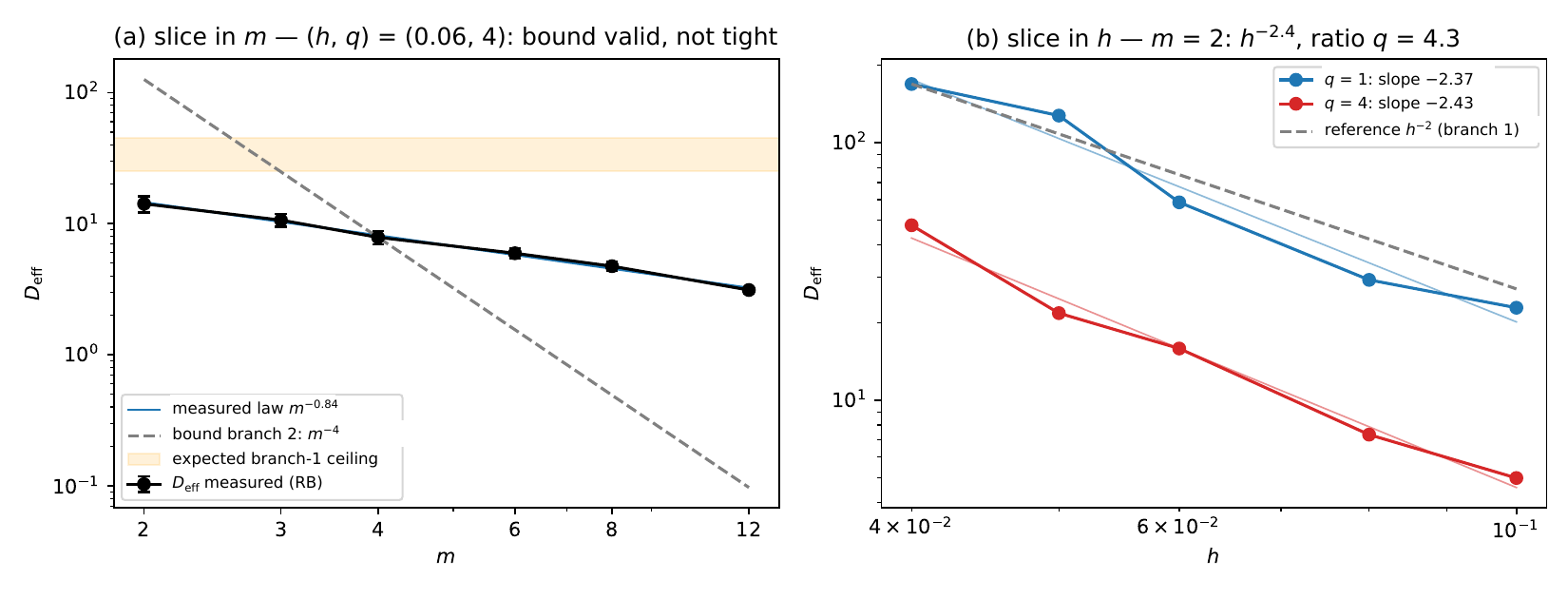}
	\caption{Effective dimension against the branches of Lemma~\textup{(U1$'$)}:
		(a) $m$-scan at $(h,q)=(0.06,4)$, measured law $m^{-0.84}$ against the
		worst-case second branch; (b) $h$-scan at $m=2$, slopes $-2.37/-2.43$ against
		the guaranteed $h^{-2}$.}
	\label{fig:deff-scans}
\end{figure}

\subsection{Scaling in $n$}\label{subsec:scaling}
The full protocol was replicated at $n\in\{500,2000,8000\}$ ($q\in\{1,8\}$,
plus $q\in\{16,32\}$ at $n=2000$; mechanistic crossings throughout; an
independent implementation, gate-validated against the $n=500$ results of
Section~\ref{subsec:exponent}). Three findings (Figure~\ref{fig:scaling}).

First, the scaling test, stated honestly: within the joint mechanistic model
the law $\mstar(1)\propto n^{1/3}$ is structural, so it cannot be ``measured''
by regressing the joint crossings. The two non-circular tests are that a single
triple $(c_V,c_B,B)=(0.73,1.62,0.98)$, free of $n$, fits all $55$ cells across
a sixteenfold range of $n$ with $R^2=0.997$ (per-scale: $0.953$, $0.966$); and
that the identifiable combinations of the per-scale free fits are stable
($c_V+c_B$: $2.57\to2.31$; $B$: $1.07\to0.88$). Off the saturated cells $c_V$
and $c_B$ carry the same regressor $(nm^2q)^{-3/5}$ and only their sum is
identified; the free-fit crossing slope over $\{500,2000\}$ is $0.384$ with a
wide interval containing $1/3$.

Second, the displacement: per-scale independent fits give
$\mstar(1)/\mstar(8)=2.83$ at both $n=500$ and $n=2000$ ($\sqrt8=2.828$); the
joint fit gives $2.81/2.83/2.83$ across the three scales, and the raw slope in
$q$ at $n=2000$ (four points crossing saturation) is $-0.446$.

Third, the floor constant: the plateau at $n=2000$, read at $q=32$, is
$\mstar=5.87$ against $\mfloor(2000)=3.90$: the ratio moves from $1.71$ at
$n=500$ to $1.51$ (CI $[1.32,1.75]$) --- in the direction of the asymptotic
equivalence of Corollary~\ref{cor:floor}, not yet significant on its own; the
per-scale fit at $n=2000$, whose floor is less well pinned, gives $1.81$, and
both values are reported.

\begin{figure}[htbp]
	\centering
	\includegraphics[width=\textwidth]{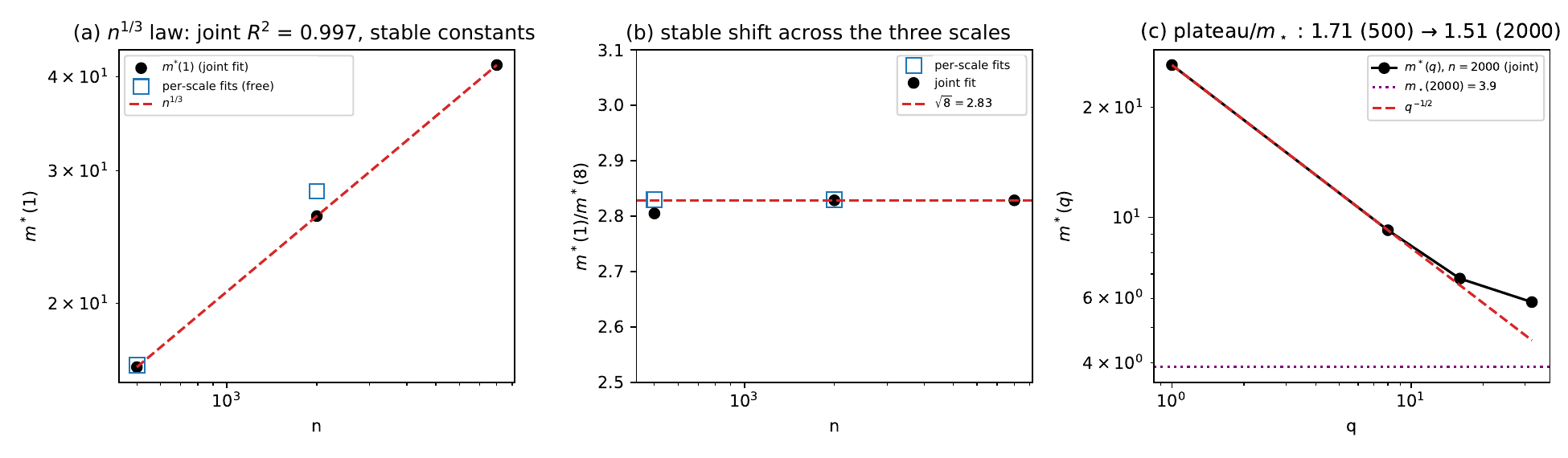}
	\caption{Scaling in $n$: (a) $\mstar(1)$ against $n^{1/3}$, joint and
		per-scale fits; (b) the ratio $\mstar(1)/\mstar(8)$ against $\sqrt8$ at the
		three scales; (c) saturation at $n=2000$ and the ratio plateau$/\mfloor$.}
	\label{fig:scaling}
\end{figure}

\begin{figure}[htbp]
	\centering
	\includegraphics[width=\textwidth]{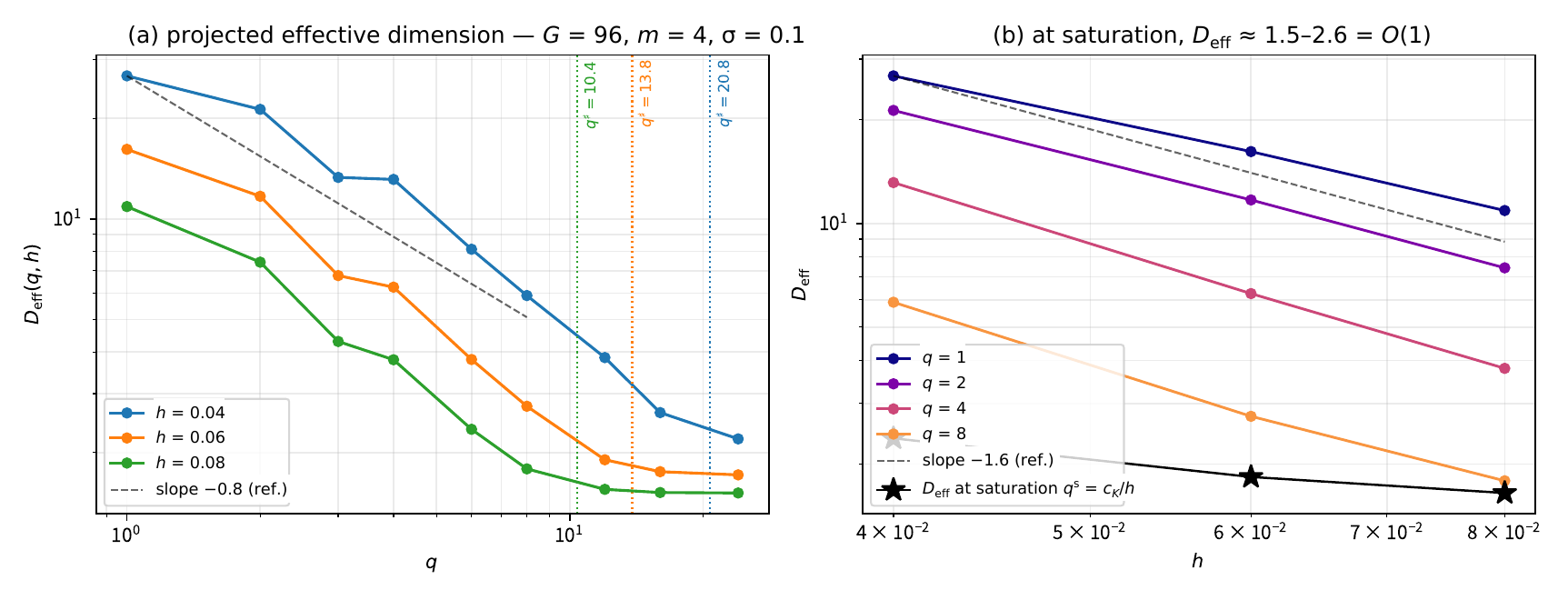}
	\caption{Effective dimension of the projected estimator covariance: power laws
		$h^{-1.6}q^{-0.8}$ below saturation (left), bounded values $\approx1.5$--$2.6$
		at $\qsh=c_K/h$ (right, stars).}
	\label{fig:deff-saturation}
\end{figure}

\section{Discussion}\label{sec:discussion}

A finite symmetry of the domain displaces the sparse--dense transition of
covariance estimation from $n^{1/(2\beta)}$ to $n^{1/(2\beta)}q^{-1/2}$, but
only until the orbit becomes finer than the bandwidth, after which no symmetry
whatsoever lowers the threshold below $n^{1/(4\beta)}$. Between the two lies a
phase diagram whose third region --- where an imposed symmetry the process does
not possess caps the risk at a level no amount of sampling can improve --- is,
in our view, the practically important one, since it is the failure mode of the
method and it has an observable signature.

Four questions are left open, in increasing order of difficulty.

The first is the gap left by Lemma~\textup{(U1$'$)}: the constant-one regime of
Corollary~\ref{cor:oracle-regimes} is proved for
$m\lesssim n^{\beta/(2\beta^2+\beta+1)}$, while the measurements of
Section~\ref{subsec:scans} indicate that the effective dimension diverges
essentially up to the threshold $\mstar(q)$, the slack lying in the
frequency-splitting bound (the shared-index configuration dominates the trace
and is discarded by the proof). Closing that gap --- and treating merely
bounded design densities --- is a self-contained second-moment problem.

The second is the geometry. Everything here is on the circle, where the orbit
separation is $1/q$, the quotient is an interval, and the symmetry spectrum is
a statement about Fourier coefficients. On a general compact manifold with a
finite isometry group the saturation lemma should survive --- Poisson summation
is replaced by a trace formula and positive definiteness is unaffected --- but
the counting that produces $\qsh$ will involve the injectivity radius of the
quotient rather than $1/q$, and the exponent $1/(4\beta)$ is unlikely to be
universal. The connection with the positive-dimensional case of
\cite{TahmasebiJegelka2023} deserves to be made precise: our saturated regime
is, in effect, the finite-group approach to their dimensional collapse.

The third is the ultra-sparse boundary $N_i\equiv2$, where each subject
contributes exactly one pair. The factor $m^2$ that carries the entire group
gain degenerates, and while Lemma~\ref{lem:saturation} is unaffected, every
threshold statement of Sections~\ref{sec:rates} and \ref{sec:saturation} loses
its meaning. Whether symmetry still helps in that regime, and in what form, is
not settled by the present analysis.

A further question was brought to light by the simulations: the diagonal bias
law for Sobolev classes. For covariances with Fourier decay
$(1+|k|)^{-(\beta+1)}$ at integer $\beta$, the integrated squared bias follows
the law $h^{2\beta-1}$, driven by a diagonal band carrying $78$--$90\%$ of its
mass (Section~\ref{subsec:warning1}). A sharp minimax theory over the Sobolev
scale --- where the diagonal band, not the bulk, sets the bias --- would
quantify what integer-regularity classes cost, and whether diagonal-corrected
smoothers recover $h^{2\beta}$.

The fourth is the group itself. We select an order within a fixed nested family
of rotation groups. Selecting among non-nested groups --- rotations against
reflections, or several incommensurate periods --- costs the ordered structure
of Proposition~\ref{prop:family} and returns one to unordered model selection
with a $\log K$ penalty. Discovering the group rather than selecting from a
list is a different problem again, and one on which the symmetry spectrum
\eqref{eq:Aq} would be the natural starting point, since it is estimable and
its profile identifies the true symmetry order.

\paragraph{Code and reproducibility}
All simulations of Section~\ref{sec:numerical} are reproducible from the
replication archive accompanying this submission: one script and one archived
seed per figure, intermediate results in \texttt{.npz} form, and a README
documenting the execution incidents encountered and their resolutions.
Estimators were validated against direct enumeration to machine precision
before production.

\appendix
\section{Self-contained proofs of two structural properties}\label{app:A}

The two lemmas below establish results cited from the companion paper
\cite{Nembe2026} in the main text; we include self-contained proofs to make
this article independent of that reference for these points.

\begin{lemma}[Group averaging is the orthogonal projection]\label{lem:orth-proj}
	The operator $\Pi_q:L^2(E\times E)\to L^2(E\times E)$ defined by
	$\Pi_qF(s,t):=q^{-1}\sum_{\ell=0}^{q-1}F(s+\ell/q,\,t+\ell/q)$ is the
	orthogonal projection onto $P_{G_q}$.
\end{lemma}

\begin{proof}
	\textbf{Image.} For any $g\in G_q$, say $g\cdot x=x+j/q$ with $0\le j<q$, and
	any $F\in L^2(E\times E)$,
	\[
	(\Pi_qF)(g\cdot s,g\cdot t)
	=\frac1q\sum_{\ell=0}^{q-1}F\Bigl(s+\tfrac{j+\ell}{q},\,
	t+\tfrac{j+\ell}{q}\Bigr)=\Pi_qF(s,t),
	\]
	since $\{j+\ell\bmod q:\ell=0,\dots,q-1\}=\{0,\dots,q-1\}$. Hence
	$\Pi_qF\in P_{G_q}$.
	
	\textbf{Idempotence.} If $F\in P_{G_q}$ then
	$F(s+\ell/q,t+\ell/q)=F(s,t)$ for all $\ell$, so $\Pi_qF=F$. In particular
	$\Pi_q\circ\Pi_q=\Pi_q$.
	
	\textbf{Self-adjointness.} Since $E=\mathbb S^1$ carries the
	translation-invariant Lebesgue measure $\mu$,
	\[
	\inner{\Pi_qF}{G}
	=\frac1q\sum_\ell\iint F(s+\ell/q,t+\ell/q)G(s,t)\,d\mu(s)\,d\mu(t)
	=\inner{F}{\Pi_qG}
	\]
	by the substitution $s\mapsto s-\ell/q$. The three properties characterise the
	orthogonal projection onto $P_{G_q}$.
\end{proof}

\begin{lemma}[Transport leaves the local pair count
	invariant]\label{lem:transport}
	Let $\psi:E\to E'$ be a $C^2$-diffeomorphism with $0<c_*\le|\psi'|\le
	c^*<\infty$, and set $h'(s')=h\,|\psi'(\psi^{-1}(s'))|$. Under
	Assumption~\ref{ass:main}, the local pair density of the transported design
	satisfies
	\[
	\iint K_{h'(s')}(s'-u')K_{h'(t')}(t'-v')f'(u')f'(v')\,du'\,dv'
	\asymp\iint K_h(s-u)K_h(t-v)f(u)f(v)\,du\,dv,
	\]
	uniformly in $s',t'\in E'$, with constants depending only on $c_*,c^*$,
	$\norm{K}_\infty$ and the bounds on $f$.
\end{lemma}

\begin{proof}
	Set $u=\psi^{-1}(u')$, $v=\psi^{-1}(v')$. The transported design density is
	$f'(u')=f(\psi^{-1}(u'))/|\psi'(\psi^{-1}(u'))|$ (change of variables for the
	push-forward of $f\,d\mu$). By the definition of
	$h'(s')=h\,|\psi'(\psi^{-1}(s'))|$,
	\[
	K_{h'(s')}(s'-u')
	=K_{h|\psi'(s)|}\bigl(\psi(s)-\psi(u)\bigr)\cdot
	|\psi'(u)|/|\psi'(s)|\cdot\bigl(1+O(h)\bigr)
	\]
	uniformly in $s,u$ with $|s-u|\le h$ (first-order Taylor of $\psi$ on the
	support of $K_h$, using $|\psi''|\le c^{**}$). Substituting and collecting the
	Jacobians $|du'|=|\psi'(u)|\,|du|$ gives the original integrand up to the
	multiplicative factor $|\psi'(u)|/|\psi'(s)|\in[c_*/c^*,c^*/c_*]$. Since $f$ is
	bounded between positive constants, the entire integral is bounded above and
	below by fixed multiples of the un-transported version.
\end{proof}

\end{document}